\documentclass[reqno]{amsart}
\usepackage{mathrsfs}
\usepackage{amsmath}
\usepackage{amsthm}
\usepackage{amsfonts}
\usepackage{amssymb}
\usepackage{url}
\usepackage{enumerate}
\usepackage[pdftex,bookmarks=true]{hyperref}
  \usepackage[usenames, dvipsnames]{color}
\usepackage{verbatim}
\usepackage{graphicx}

\usepackage{pdfsync}
\usepackage{bbm}

\newcommand\Vol{{\operatorname{Vol}}}

\newcommand\R{{\mathbb{R}}}

\renewcommand\P{{\mathbb{P}}}
\newcommand\E{{\mathbb{E}}}

\newcommand\Var{{\operatorname{Var}}}

\newcommand\Z{{\mathbf{Z}}}

\newcommand\bxi{\boldsymbol{\xi}}

\newcommand\var{{ \operatorname{Var}}}

\newcommand\eps{\varepsilon}

\newcommand\Ba{{\mathbf a}}
\newcommand\Bb{{\mathbf b}}

\newcommand\Be{{\mathbf e}}
\newcommand\Bf{{\mathbf f}}
\newcommand\Bg{{\mathbf g}}

\newcommand\Bs{{\mathbf s}}

\newcommand\Bu{{\mathbf u}}
\newcommand\Bv{{\mathbf v}}

\newcommand\Bx{{\mathbf x}}
\newcommand\By{{\mathbf y}}

\newcommand\BN{{\mathbf N}}

\newcommand\BS{{\mathbf S}}
\newcommand\BT{{\mathbf T}}

\newcommand\CA{{\mathcal A}}

\newcommand\CE{{\mathcal E}}

\newcommand\CH{{\mathcal H}}

\newcommand\CK{{\mathcal K}}

\newcommand\CP{{\mathcal P}}

\newcommand\CS{{\mathcal S}}

\newcommand\CU{{\mathcal U}}

\newcommand{\al}{\alpha}
\newcommand{\la}{\lambda}
\newcommand{\Ent}{\operatorname{Ent}}

\newcommand{\1}{\mathbbm{1}}

\usepackage{fullpage}

\theoremstyle{plain}
  \newtheorem{theorem}[subsection]{Theorem}
  
  \newtheorem{problem}[subsection]{Problem}

    \newtheorem{proposition}[subsection]{Proposition}
  
  \newtheorem{lemma}[subsection]{Lemma}
  \newtheorem{corollary}[subsection]{Corollary}
  \newtheorem{cor}[subsection]{Corollary}
 \newtheorem{question}[subsection]{Question}

  \newtheorem{remark}[subsection]{Remark}
  
  \newtheorem{claim}[subsection]{Claim}

\theoremstyle{definition}

\begin{document}

\title[Concentration of the number of real roots]{Concentration of the number of real roots of random polynomials}

\author{Ander Aguirre}
\address{Department of Mathematics\\ The Ohio State University \\ 231 W 18th Ave \\ Columbus, OH 43210, USA}
\email{aguirre.93@osu.edu}

\author{Hoi H. Nguyen}
\email{nguyen.1261@osu.edu}

\author{Jingheng Wang}
\email{wang.14053@osu.edu}

\thanks{The authors are supported by National Science Foundation CAREER grant DMS-1752345.}



\maketitle
\begin{abstract} Many statistics of roots of random polynomials have been studied in the literature, but not much is known on the concentration aspect. In this note we present a systematic study of this question, aiming towards nearly optimal bounds to some extent. Our method is elementary and works well for many models of random polynomials, with gaussian or non-gaussian coefficients. 

\end{abstract}


\section{Introduction} Over the past few years, there have been active developments to study various statistics of the number $N_\R(F)$ of real roots of a random polynomial $F$. In its general setting, the random polynomial $F$ takes the form
\begin{equation}\label{eqn:f:general}
F(x)=\sum_{j=0}^n \xi_j p_j(x),
\end{equation}
where $\xi_j$ are iid copies of a random variable $\xi$ of mean zero and variance one, and $p_j(x)$ are deterministic polynomials (of degree $j$) coming from various natural sources. The most typical choices of $\xi_j$ are standard gaussian, where we will call $F$ gaussian polynomial. Below we list a few typical examples of $F$

\begin{itemize}
\item Kac polynomial: $p_j(x)=x^j$; 
\vskip .05in
\item Trigonometric polynomials: $p_j(x) =\cos (jx)$ or $\sin(jx)$ or a combination of both; 
\vskip .05in
\item Orthogonal polynomials: $p_j(x)$ is a polynomial of degree $j$ and $\{p_j(x)\}_{j=0}^n$ forms an orthonormal basis with respect to a smooth Borel measure $\mu$ on $\R$;  
\vskip .05in
\item Elliptic polynomial: $p_j(x) = \sqrt{\binom{n}{j}}x^j$;
\vskip .05in
\item Weyl polynomial: $p_j(x) = \frac{1}{\sqrt{j!}} x^j$.
\end{itemize}
We refer the reader to the text \cite{book} for discussions of complex roots of these models, among many others. 
\subsection{Statistics for real roots of the gaussian model}

 We now cite several results from the literature for the gaussian models. Here, the moments statistics can be studied via the celebrated Kac-Rice formula. 
 
 For Kac's polynomials  $F(x) =\sum_{j=0}^n \xi_j x^j$, it was shown by Maslova and many others (see for instance \cite{Mas1,Mas2, OV-CLT}) that 
 
 \begin{theorem}\label{thm:Kac:gau}
 $$\E N_\R = (\frac{2}{\pi}+o(1)) \log n \mbox{ and } \Var(N_\R) = ( \frac{4}{\pi} (1 -\frac{2}{\pi})+o(1)) \log n$$
and moreover
$$\frac{N_\R - \E N_\R}{\sqrt{\Var(N_\R)}} \xrightarrow{d} \BN(0,1).$$
 \end{theorem}

For trigonometric polynomial $F(x) =\frac{1}{\sqrt{n}} \sum_{j=0}^n a_j \cos(jx) + b_j \sin(jx)$, the expected value of $N_\R(F)$ was shown in \cite{Qualls} and the variance and Central Limit Theorem were established in \cite{ADL,AL, GW}.  

\begin{theorem}\label{thm:trig:gau} We have 
  $$\E N_\R =  (2/\sqrt{3} +o(1)) n \mbox{ and } \Var(N_\R) = (K+o(1)) n$$
where $K$ is an explicit positive constant and moreover
$$\frac{N_\R - \E N_\R}{\sqrt{\Var(N_\R)}}  \xrightarrow{d} \BN(0,1).$$
 \end{theorem}

Next, for Weyl polynomial $F(x) =  \sum_{j=0}^n  \xi_j \frac{1}{\sqrt{j!}} x^j$ it was shown in \cite{DV} that 

 \begin{theorem}\label{thm:Weyl:gau} For gaussian Weyl polynomials, we have 
$$\E N_\R = (\frac{2}{\pi}+o(1))n \mbox{ and } \Var(N_\R) = (K+o(1))n$$
where $K$ is an explicit positive constant, and moreover
$$\frac{N_\R - \E N_\R}{\sqrt{\Var(N_\R)}}  \xrightarrow{d} \BN(0,1).$$
\end{theorem}
 
For Elliptic polynomials $F(x)=\sum_{j=0}^n \xi_j \sqrt{\binom{n}{j}} x^j$, the expected value was computed in \cite{K} while the variance and Central Limit Theorem were established in \cite{Dal} (see also \cite{AL}). 
 \begin{theorem}\label{thm:E:gau} For gaussian Elliptic polynomials, we have 
$$\E N_\R = \sqrt{n} \mbox{ and } \Var(N_\R) = (K+o(1))\sqrt{n}$$
where $K$ is an explicit positive constant, and moreover
$$\frac{N_\R - \E N_\R}{\sqrt{\Var(N_\R)}}  \xrightarrow{d} \BN(0,1).$$
\end{theorem}  
Now for orthogonal polynomials, the mean was computed in \cite{LPX} and the variance was computed in \cite{LP}.

\begin{theorem}\label{thm:O:gau} Let $\mu$ be a measure with compact support on the real line, that is regular in the sense of Stahl, Totik, and Ullmann. Let $\omega$ denote the Radon-Nikodym derivative of the equilibrium measure for the support of $\mu$. Let $[a', b']$ be a subinterval in the support of $\mu$, such that $\mu$ is absolutely continuous there, and its Radon-Nikodym derivative $\mu'$ is positive and continuous there. Assume moreover, that
$$\sup_{n\ge 1} \|p_n\|_{L_\infty[a',b']}<\infty.$$
Then if $[a,b] \subset (a',b')$ we have 
$$\E N_{[a,b]} =  \left(\frac{\nu_\CK([a,b])}{\sqrt{3}}+o(1)\right)n,$$ 
where $\nu_\CK$ is the equilibrium measure of the support $\CK$ of $\mu$ (i.e. $\nu_\CK$ minimizes $I[\nu] = - \int \int  \log|z-t| d \nu(t) d\nu(z)$ among all probability measures $\nu$ with support on $\CK$.)\

Furthermore,
$$\lim_{n\to \infty}\frac{1}{n} \Var(N_{[a,b]}) = c \int_{a}^b \omega(y) dy$$
where $c$ is an explicit absolute positive constant (independent of $\CK, \mu$.) Finally, it is shown in \cite{DLNgNgP} that $N([a,b])$ has CLT fluctuation.
\end{theorem}

For the sake of a heuristic motivation of concentration, let us review the correlation formulae and one point intensities. For a random polynomial $F$ of degree $n$, denote the covariance  and normalized covariance kernel by 
$$K_n(s,t)=\E F(s)F(t);\quad \quad \quad  \overline{K}_n(s,t)=\E \frac{F(s)F(t)}{\sqrt{\E F^2(s) \E F^2(t)}}.$$
 Then for the aforementioned models one has (see \cite{AL} and \cite{SM}):

\begin{itemize}
     \item For Kac polynomials $F(x) =\sum_{j=0}^n \xi_j x^j$
     \begin{align*}
        \overline{K}_n(s,t) =\frac{1}{\sqrt{(\sum_{j=0}^n s^{2j}) (\sum_{j=0}^n t^{2j}) }} \frac{1-(st)^{n+1}}{1-st};
     \end{align*}

       \item For (stationary) trigonometric polynomials $F(x)=\frac{1}{\sqrt{n}}\sum_{j=0}^n \xi_{j1} \cos(jx) + \xi_{j2} \sin(jx)$,

   \begin{align*}
          \overline{K}_n(s,t)=\sum_{j=0}^n\frac{\cos(j(s-t))}{n}=\frac{1}{n}\cos\left(\frac{(n+1)(s-t)}{2}\right)\frac{\sin\left(\frac{n(s-t)}{2}\right)}{\sin\left(\frac{(s-t)}{2}\right)};
    \end{align*}

   \item For Elliptic polynomials $F(x) =  \sum_{j=0}^n  \xi_j \sqrt{\binom{n}{j}}  x^j$
    \begin{align*}
         \overline{K}_n(s,t)= \frac{1}{\sqrt{(1+s^2)^n (1+t^2)^n}}(1+st)^n=\left(\frac{1+st}{\sqrt{(1+st)^2+(s-t)^2}}\right)^n;
    \end{align*}
   \item For Weyl polynomials $F(x) = \sum_{j=0}^n  \xi_j \frac{1}{\sqrt{j!}} x^j$
   
    \begin{align*}
           \overline{K}_n(s,t)=\frac{1}{\sqrt{\sum_{j=0}^n\frac{s^{2j}}{j!} \sum_{j=0}^n\frac{t^{2j}}{j!}}}\sum_{j=0}^n\frac{(st)^j}{j!}.
     \end{align*}

\end{itemize}
Under 
appropriate scalings one can see that 
\begin{equation}\label{eqn:K_n}
\overline{K}_n(s,t) \to 0 \mbox{ as } |s-t|\to \infty.
\end{equation}
For instance for trigonometric polynomials, we have that under the rescaling $x\rightarrow \frac{x}{n}$ the normalized covariance Kernel in the $n\rightarrow \infty$ limit converges to (see for instance \cite[Lemma 1]{ADL}) $\operatorname{sinc}(s-t)=\frac{\sin(s-t)}{(s-t)}$, which converges to zero when $|s-t| \to \infty$. In the orthogonal model, explicit formulas are more complicated but the limiting normalized covariance Kernel after  the rescaling $x\rightarrow \frac{x}{n}$ also yields the sinc kernel (as a consequence of \cite[Lemma 3.3]{LP}). For the Elliptic polynomial, in the regime $s,t =O(\sqrt{n})$ (where one expect the most number of real roots, see the first intensity $\rho_1(x)$ discussed below) we see that $ \overline{K}_n(s,t) \to 0$ as $|s-t| \to \infty$ with $n$. For Weyl polynomial, also in the regime $s,t =O(\sqrt{n})$ (where one expect the most number of real roots, see its density $\rho_1(\cdot)$ below)
we see that $ \overline{K}_n(s,t) \approx \exp(-(s-t)^2/2)\to 0$ as $|s-t| \to \infty$. The property \eqref{eqn:K_n} can be used as an indicator that $N_I$ and $N_I'$ are approximately independent if $I$ and $I'$ are far apart.

 We also remark that for Gaussian coefficients the computation of the first intensity of the number of real roots via  Kac-Rice reduces to the simple formula
\begin{align*}
    \rho_1(x)=\frac{1}{\pi}\sqrt{\frac{\partial^2}{\partial s \partial t}\log K_n(s,t)|_{s=t=x}}.
\end{align*}
For instance we have the following:
\begin{itemize}

    \item For Kac polynomials (see for instance \cite{EK})
     \begin{align*}
        \rho_1(x)=\frac{1}{\pi}\sqrt{\frac{1}{(x^2-1)^2}-\frac{(n+1)^2x^{2n}}{(x^{2n+2}-1)^2}};
    \end{align*}
  \item For (stationary) trigonometric polynomials on $[0,2\pi]$ (see for instance \cite{AL, Qualls})
   \begin{align*}
          \rho_1(x)=  \frac{1}{\pi}\sqrt{(n+1)(2n+1)/6}; 
    \end{align*}
  \item For Weyl polynomials (see for instance \cite{SM}) 
    \begin{align*} 
          \rho_1(x) = \frac{1}{\pi} \sqrt{ 1 + \frac{x^{2n} (x^2-n-1)}{e^{x^2} \Gamma(n+1,x^2)}  - \frac{x^{4n+2}}{[e^{x^2} \Gamma(n+1, x^2)]^2}},
     \end{align*}
    where $\Gamma(n,x) = \int_x^\infty e^{-t} t^{n-1} dt$, from which we see that $\rho_1(x) \approx \frac{1}{\pi}$ for $|x| \le (1-o(1))\sqrt{n}$ and $\rho_1(x) \approx \sqrt{n}/\pi x^2$ if $x \ge (1+o(1))\sqrt{n}$;
    \vskip 5mm
    \item For Elliptic polynomials (see for instance \cite{EK}) 
    \begin{align*} 
          \rho_1(x)=\frac{\sqrt{n}}{\pi (1+x^2)}.
    \end{align*}
\end{itemize}

\subsection{Universality} As we have seen, for gaussian polynomials most of the statistics have been known. Another fascinating aspect of the theory of random polynomials is its universality in terms of random coefficients. These have been verified via recent works by Kabluchk-Zaporozhets \cite{KZ1}, Tao-Vu \cite{TV}, Do-O.Nguyen-Vu \cite{DONgV1}, and O. Nguyen-Vu \cite{OV} and by many others.  Among other things, it was shown that the root correlations of the ensembles above (Weyl, Elliptic, Orthogonal) where  $\xi_j$ are iid copies of a random variable $\xi$ of mean zero, variance one and bounded $(2+\eps)$-moment are asymptotically the same as in the gaussian case. We then deduce from these statements the following results.

 \begin{theorem}\cite{DONgV1, TV}\label{thm:E:gen} For Elliptic polynomials where $\xi_j$ are iid copies of a random variable $\xi$ of mean zero, variance one and bounded $(2+\eps)$-moment, we have 
$$\E N_\R = (\frac{2}{\pi}+o(1))\sqrt{n}, \mbox{ and } \Var(N_\R) =O(n^{1-c}),$$
where $c>0$ is a positive constant.
\end{theorem}  

\begin{theorem}\cite{DONgV1,TV}\label{thm:W:gen} For Weyl polynomials where $\xi_j$ are iid copies of a random variable $\xi$ of mean zero, variance one and bounded $(2+\eps)$-moment, we have 
$$\E N_\R = (1+o(1))\sqrt{n}, \mbox{ and } \Var(N_\R) =O(n^{1-c}),$$
where $c>0$ is a positive constant.
\end{theorem}  
We pause to remark the following, which will be useful later.
\begin{remark}\label{rmk:restriction}
It follows immediately from the proofs of  \cite{DONgV1} and \cite{TV} that the variance estimates of Theorem \ref{thm:E:gen} and Theorem \ref{thm:W:gen}  continue to hold if one replace $N_\R$ by $N_I$ for any interval $I$ where $\E N_I$ has order $\sqrt{n}$.
\end{remark}

For orthogonal polynomials, the following was shown in \cite{DONgV2} (see also \cite{DLNgNgP})

\begin{theorem}\label{thm:O:gen} For orthogonal polynomials, assume that $F_n(x)= \sum_{j=0}^n \xi_j  p_j(x)$ where $\xi_j$ are iid copies of a random variable $\xi$ of mean zero, variance one and bounded $(2+\eps)$-moment and $p_n$ is the system of orthogonal polynomials with respect to $\mu$ as in Theorem \ref{thm:O:gau}. Then 
$$\E N_{[a,b]} =  \left(\frac{\nu_\CK([a',b'])}{\sqrt{3}}+o(1)\right)n,$$
and
$$\Var(N_{[a,b]}) =O(n^{2-c}),$$
where $c>0$ is a positive constant.
 \end{theorem}
There remain interesting questions left unanswered.
\begin{problem}[Universality of fluctuation] With the aim to keep $\xi_i$ as general as possible \footnote{For instance we only assume $\xi_i$ to be subgaussian of mean zero, and variance one.},
\begin{enumerate}
\item Is it true that the variances in Theorem \ref{thm:E:gen} and Theorem \ref{thm:W:gen} are of order $O(\sqrt{n})$, and in Theorem \ref{thm:O:gen} is of order $O(n)$?
\vskip .1in
\item Is it true that the number of real roots (for any of the above models) obeys CLT fluctuation?
\end{enumerate}
\end{problem}
As far as we are concerned, the above problems are known only for Kac polynomials via the works of Maslova and O. Nguyen-Vu \cite{Mas1,Mas2, OV-CLT} (see also \cite{DNh}). For trigonometric polynomials it is known from  \cite{BCP, DNgNg} that the variance has linear order $O(n)$ with multiplicative constant depending on the fourth moment of $\xi$.


\section{A new direction: concentration of the statistics} In this note, we focus on another important aspect of the number of real roots $N_I$ (where for instance $I=\R$ in the case of Weyl and Elliptic polynomials, and $I=[a,b]$ in the case of orthogonal polynomials). 
\begin{question}
How well is $N_I$ concentrated around its mean?
\end{question}
In fact, the study of this problem dates back at least to the very early work of Erd\H{o}s-Offord \cite{EO} in the 1940s when they studied the concentration of the number of real roots for Kac polynomials. In the 1970s, Dunnage \cite{Dun} and Qualls \cite{Qualls} also raised the question of concentration for the number of roots of random trigonometric polynomials. 

First, it follows from  Theorems \ref{thm:E:gen}, \ref{thm:W:gen}, \ref{thm:O:gen}  for the Elliptic, Weyl, and orthogonal models that $N_T$ are well concentrated around its mean with high probability, that is
\begin{equation}\label{eqn:Markiv}
P(|N_I-\E N_I|\ge \la \E N_I) \le \frac{1}{n^c \la^2}
\end{equation}
for some positive constant $c$. 


Heuristically, at least for the Gaussian case or when the common distribution $\xi$ has smooth density, for any large interval $I$ (such that $M:= \E N_I \to \infty$) we can write $N_I = \sum_{i=1}^{M} N_i$, where $N_i$ is the number of roots over $I_i$, an interval of length $O(1/M)$. Following from the correlation $\overline{K_n} \to 0$ as $|s-t| \to \infty$ discussed in \eqref{eqn:K_n}, the $N_i$ seem to be weakly independent, which would suggest that $N_I$ has exponentially concentration (with exponent $M$) around the mean (and also the prediction that $\var(N_I) = O(M)$ as well as the CLT fluctuation of $N_I$). However, the justification of the above heuristic does not seem to be easy; as far as we know, the exponential concentration and CLT fluctuation problems are only resolved in the limiting (stationary) gaussian process case by the result of Cuzick \cite{Cuz} and Basu et. al. \cite{BDFZ} (to be mentioned below). 


We next list several results from the literature. The first result, which has significant influence over our technique is the result from \cite{NS} by Nazarov and Sodin. Let 
$$F= \sum_{k=-n}^n \xi_k Y_k,$$ 
where $\xi_k$ are iid gaussian, $\{Y_k\}_{-n}^n$ are is an orthonormal basis of $\CH_n$, the 
$(2n+1)$-dimensional real Hilbert space of spherical harmonics of degree $n$ on the $2$-dimensional unit sphere.  Let $N(f)$ denote the  number of connected components of $Z(f) =\{x\in S^2, f(x)=0\}$.

\begin{theorem}\label{thm:NS} There exists a constant $a>0$ such that for every $\eps>0$ we have 
$$\P(|\frac{N(F)}{n^2} -a| >\eps) \le C(\eps) e^{-c(\eps)n}.$$
\end{theorem}
In another direction, Rozenshein \cite{Rozen} consider the following variant of arithmetic random waves. Let $d\ge 2$ be fixed and 
$\CH_L=\operatorname{span}\big (\cos(2\pi \la \cdot x), \sin (2\pi \la \cdot x), \la \in \Lambda_L =\{\la \in \Z^d, |\la|=L\}\big).$ 
Consider 
$$F(x) = \sqrt{\frac{2}{\dim \CH_L}} \sum_{\la\in \Lambda_L^+} a_\la \cos(2\pi \la \cdot x) + b_\la \sin (2\pi \la \cdot x).$$
Let $N(F)$ denote the number of nodal components of $f$.  
\begin{theorem}\cite[Theorem 1.3]{Rozen}\label{thm:Rozen} For every $\eps>0$ we have 
$$\P(|\frac{N(F)}{L^d} - m(\frac{N(F)}{L^d})| >\eps) \le C(\eps) e^{-c(\eps)\dim \CH_L}.$$
\end{theorem}
Another result is for the Kostlan model. Consider 
$$F(x) = \sum_\al \xi_\al (\frac{n!}{\al_0! \dots \al_d!})^{1/2} x_0^{\al_0}\dots x_d^{\al_d},$$
where $\xi_\al$ are iid gaussian. When $d=2$, we obtain our Elliptic model. In this case the following is from \cite[Remark 9]{DL} and \cite[Corollary 1]{GW}\footnote{Although these results focus more on Betti numbers and on higher dimension.}.
\begin{theorem}\label{thm:E:gau} There exist constants $C,D$ such that for any $\tau=\tau(n)$
$$\P(N_\R \ge \tau \sqrt{n}) \le Cn^{3/2} \exp(-D \tau^2).$$
\end{theorem}
Note that the above result is a large deviation result, being effective only when $\tau \gg \sqrt{\log n}$. 

In another direction, among other things the following was shown in \cite{BDFZ}.

\begin{theorem}\label{thm:process} Suppose the centered stationary Gaussian process $\{X(t), t\in \R\}$ of a compactly supported spectral measure has an integrable covariance. Then for some $C<\infty$ and $c(\cdot)>0$, 
$$\P(|N_{[0,T]} - \E N_{[0,T]} | \ge \eta T) \le C e^{-c(\eta)T}, \forall \eta>0, T<\infty.$$
\end{theorem}

As the reader can see, the concentration results of Theorem \ref{thm:NS} and Theorem \ref{thm:Rozen} (and partly of Theorem \ref{thm:E:gau}) are in high dimensional settings -- it is desirable to establish similar results for the univariate case. Indeed, this problem was raised in \cite[Subsection 1.6.3]{NS-general} as an interesting direction. Partly motivated by this, the second author and Zeitouni initiated a study to show concentration that works for a wide range of randomness. Consider a random trigonometric polynomial of degree $n$
\begin{equation}\label{eqn:P_n}
F(x) =\frac{1}{\sqrt{n}} \sum_{k=1}^n a_k \cos(kx) + b_k \sin(kx),
\end{equation}
where $a_k, b_k$ are i.i.d. copies of a random variable $\xi$ of mean zero and variance one. Let $N_n$ denote the number of roots of $F_n(x)$ for $x\in [-\pi,\pi]$.

\begin{theorem}\cite{NgZ}\label{thm:trig} Let $C_0$ be a given positive constant, and suppose that either $|\xi|$ is bounded almost surely by $C_0$, or 
that its law satisfies the logarithmic Sobolev inequality with parameter $C_0$. Then there exist constants $c,c'$ such that for $\eps\geq n^{-c} $ we have that
$$\P(  |N_{[0,2\pi]} - \E N_{[0,2\pi]}| \ge \eps n ) \le e^{ - c' \eps^9 n}.$$
\end{theorem}




\subsection{New results} 


It is non-trivial to use Kac-Rice formula to compute the right order of $\Var(N_I)$ even in the gaussian case. There have been works to compute higher order moments (such as \cite{AL,AP', LGass} and the references therein) but from these results, one can only obtain polynomial concentration at best. Basing on the heuristic discussed previously, we would expect sub-exponential concentration for most models. In this note, basing on our initial works from \cite{NgZ, Ngcurve}, we will give a systematic method to deal with such problem, and at the same time prove concentration results for many other models of random functions. There are various advantages, that as our approach is based on isoperimetric inequality, it works for fairly general randomness. In fact in some cases our result also works for random variables of bounded higher moments, and we will comment on that later. Now we mention our main results.

\begin{theorem}[Concentration for orthogonal polynomials]\label{thm:orthogonal} Assume as in Theorem \ref{thm:O:gau} and suppose that for some positive constant $C_0$ either $1/C_0 < |\xi|<C_0$ with probability one, or that $\xi$ is continuous with bounded density and satisfies the logarithmic Sobolev inequality \eqref{eqn:logSobolev} with parameter $C_0$. Then  there exist constants $c,D$ such that for $n^{-c} \le \eps \le 1$
$$\P\Big(  |N_n([a,b]) -\E N_n([a,b])| \ge \eps n \Big) \le e^{ -\eps^{D} n}.$$
\end{theorem}
Note that the exponent is nearly optimal, as in the Bernoulli case ($\xi_i$ taking values $\pm 1$ independently with probability 1/2) one cannot get anything below $2^{-(n+1)}$. Notably our current result extends Theorem \ref{thm:trig} to a number of natural families of orthogonal polynomials (where we also invite the reader to \cite{LP,DLNgNgP} for many more examples), for instance when the systems of $(p_i(x))$ are
	\begin{itemize}
		\item Legendre polynomials: the weight $w(x)$ is the indicator function of $[-1,1]$.\\
		\item Chebyshev polynomials (type 1):  $w(x)=\frac 1{\sqrt{1-x^2}}$ for $x\in (-1,1)$.\\
		\item Chebyshev polynomials (type 2):  $w(x)=\sqrt{1-x^2}$ for $x\in (-1,1)$.\\
		\item Jacobi polynomials: these are generalizations of Chebyshev and Legendre polynomials, with $w(x)=(1-x)^\beta (1+x)^\gamma$ for $x\in (-1,1)$. The parameters $\beta,\gamma$ are constants larger than $-1$.
			\end{itemize}

By using a similar method, but with totally different technical details, for Elliptic polynomials we show
\begin{theorem}[Concentration for Elliptic polynomials]\label{thm:Elliptic} Let $C_0$ be a given positive constant and suppose $\xi$ is continuous with bounded density and satisfies the logarithmic Sobolev inequality \eqref{eqn:logSobolev} with parameter $C_0$. Let $C_E>0$ be given, and let $I = I_E = [- C_E\sqrt{n}, C_E \sqrt{n}]$. Then  there exist constants $c,D_1,D_2$ such that for $n^{-c} \le \eps \le 1$, for Elliptic polynomials we have
$$\P(|N_I - \E N_I | \ge \eps \sqrt{n}) \le e^{- (\eps^{D_1}/\log^{D_2} n)  \sqrt{n}}.$$
\end{theorem}

The exponent here is nearly optimal because it was shown in \cite{DM} (confirming various predictions from \cite{SM}) that the probability that $N_\R=0$ (persistence probability) in the gaussian case has order $\exp(-\Theta(\sqrt{n}))$. We also remark that in the gaussian case our result compensates the large deviation type result of Theorem \ref{thm:E:gau} (and the results of  \cite{GaWe} in one dimension).

For Weyl polynomials we are able to show the following analogous result.

\begin{theorem}[Concentration for Weyl polynomials] \label{thm:Weyl} Let $C_0$ be a given positive constant and suppose $\xi$ is continuous with bounded density and satisfies the logarithmic Sobolev inequality \eqref{eqn:logSobolev} with parameter $C_0$. Let $C_W>1$ 
be given, and let $I = I_W = [ \frac{1}{C_W}\sqrt{n}, C_W \sqrt{n}]$ or $I = [-C_W \sqrt{n},  -\frac{1}{C_W}\sqrt{n}]$. Then  there exist constants $c,D_1,D_2$ such that for $n^{-c} \le \eps \le 1$ and for Weyl polynomials we have
$$\P(|N_I - \E N_I | \ge \eps \sqrt{n}) \le e^{- (\eps^{D_1}/\log^{D_2} n)  \sqrt{n}}.$$
\end{theorem}

The exponent here is also nearly optimal, as here too the persistence probability $N_\R=0$ for the gaussian case has been shown by \cite{DM, CP} to have order $\exp(-\Theta(\sqrt{n}))$.

We note that all of the constants 
$D_1,D_2$ from Theorem \ref{thm:Elliptic} and Theorem \ref{thm:Weyl} above can be made explicit. We can also allow $C_E,C_W$ to vary with $n$ (such as $C_E, C_W = n^{c}$ for some small constant $c$, see the remark at the end of Theorem \ref{thm:repulsion:E} and Lemma \ref{lemma:W:prob:d}) but let us keep $C_E,C_W$ to be fixed here, noting that the range  $[- C\sqrt{n}, C\sqrt{n}]$ or $[-C \sqrt{n}, -\frac{1}{C}\sqrt{n}]$ or $[\frac{1}{C}\sqrt{n}, C\sqrt{n}]$ for sufficiently large $C$ are where the majority of the roots belong to. (Which can be seen from the limiting first intensity of the roots of Elliptic polynomials and Weyl polynomials mentioned previously.) We also remark that nearly optimal concentration for real roots of Kac polynomials has been established recently by O. Nguyen and Can in \cite{CNg}.

We have to assume $\xi$ to satisfy the log Sobolev in both Theorem \ref{thm:Elliptic} and Theorem \ref{thm:Weyl} in order to apply geometric concentration for log-Sobolev random variables (such as Theorem \ref{thm:sobolev} and Theorem \ref{thm:E:concentration}). The proofs do not seem to extend to more general random variables. However in this general case we can show the following deviation result.

\begin{theorem}\label{thm:dev} Suppose that $\xi$ has mean zero, variance one, and bounded $(2+\eps)$-moment. Assume that $\log^4 n \ll  A \ll \sqrt{n}$. Then there exists a constant $c$ such that for both the Elliptic and Weyl polynomials (where $I=I_E$ or $I_W$ accordingly),
$$\P(N_I \ge A \sqrt{n}) \le 2^{-c A} + \exp(-\Theta(\sqrt{n}/\log^2 n)).$$
\end{theorem}

\subsection{
Further remarks}

First, together with the estimates from Theorem \ref{thm:O:gen}, Theorem \ref{thm:E:gen} and Theorem \ref{thm:W:gen}, our results immediately imply that
$$\frac{N_I}{\E N_I} \to 1, \mbox{ almost surely}$$
for all three models, where $I$ is either $[a,b]$ or $I_E$ or $I_W$ depending on the model. (See also \cite{AP'} for similar results regarding random trigonometric polynomials.) 

Second, our results also imply very fine (sub)exponential-type estimates for the persistence probability that $F$ does not have any root, i.e. $\P(N_I =0)$, in the {\it non-gaussian} cases. Such probabilities are well-studied for Gaussian processes and gaussian polynomials, see for instance \cite{AS,CP,  DM, DM', DPSZ, FFN,FFN',KK,SM} and the rich references therein. This problem has physical applications in models of Statistical Mechanics and  PDEs, for instance \cite{SM} discusses the persistence probability $p_0$ (i.e of the scalar field $\phi$ remaining positive) of the standard diffusion equation $\partial_t\phi(x,t)=\nabla^2\phi(x,t)$. With \textit{random} Gaussian initial configuration for a system of linear size $L$, and under the scaling limit $t,L \to \infty$ and $t/L^2$ fixed, an exponential-decay relation (dependent on  dimension) $p_0(t,L)\propto L^{-\theta(d)}$ was predicted.

Our paper provides a robust method to obtain strong concentration for the number of real roots of many models of random polynomials. Within this topic there remains many interesting problems that our method does not apply to (1) What is the optimal power of $\eps$ for the exponents for each of the models (even in the gaussian case)?  (2) Can one extend to higher dimensional settings similarly to Theorems \ref{thm:NS}, \ref{thm:Rozen} and  \ref{thm:E:gau}? (3) Does the method extend to gaussian processes (and then non-gaussian processes) yielding an alternative proof of Theorem \ref{thm:process}?

{\bf Notations.}  We consider $n$ as
an asymptotic parameter going to infinity and allow all other quantities to depend on $n$ unless they are
explicitly declared to be fixed or constant.  As used earlier, we write $X =
O(Y)$, $Y=\Omega(X)$, $X \ll Y$, or $Y \gg X$ if $|X| \leq CY$
for some fixed $C$; this $C$ can depend on
other fixed quantities such as the 
parameter $C_0$ in the condition of $\xi$.  If $X\ll Y$ and $Y\ll X$, we say that $Y = \Theta(X)$ or $X \asymp Y$.

\section{Supporting ingredients and proof methods}\label{section:supporting} 

In this section, 
we gather a few ingredients to be used; most of them are of independent interest. Here $T= [0,1]$ and $T'= [-\eps_0 ,1+\eps_0]$. The parameter $N$ here will be sent to $\infty$, which will be either $n$ or $\sqrt{n}$ in our later applications. 

\subsection{Large Sieve inequalities} Our first ingredient is the following variant of the large sieve inequalities, which is mostly used for trigonometric polynomials.

\begin{lemma}[Restricted large sieve inequality]\label{lemma:largesieve} Let $d\ge 1$. Let $f$ be a smooth real valued function over in interval $T=[0,1]$, and that for $k=d-1, d$ 
\begin{equation}\label{eqn:f'f}
\int_{T} (f^{(k)}(x))^2 dx \le  (C^\ast N)^{2k} S
\end{equation}
for some parameters $C^\ast$ and $S>0$ (that might depend on $f$). 
Then for any dividing points $0\le x_1 <x_2 <\dots <x_M\le 1$, with $\delta$ being the minimum of the gaps between $x_i, x_{i+1}$, we have
$$\sum_{i=1}^M (f^{(d-1)}(x_i))^2 \le  (C^\ast N+\delta^{-1}) (C^\ast N)^{2(d-1)}S.$$
\end{lemma}


\begin{proof}(of Lemma \ref{lemma:largesieve}) It suffices to assume that $\delta\le x_1$ and $x_M \le 1-\delta$. We follow the classical approach by Gallagher \cite{Ga}.
\begin{claim} Let $g$ be a differentiable function on $I=[c-h,c+h]$. Then 
$$g(c) \le \frac{1}{|I|}\int_I |g(t)|dt + \frac{1}{2}\int_I |g'(t)|dt.$$
\end{claim}
\begin{proof} Let $\rho(t) = t-(c-h)$ if $t\in (c-h,c)$ and $\rho(t) = t-(c+h)$ if $t\in (c,c+h)$. Partial integrals 
(over $(c-h,c)$ and $(c,c+h)$) give
$$\int_I \rho(t) g'(t) dt =2h g(c) - \int_{I} g(t)dt.$$
Note that $|\rho(\cdot)|\le h$, 
so the claim follows by triangle inequality.
\end{proof}
By this claim, 
\begin{align*}
\sum_{i=1}^M |f^{(d-1)}(x_i)|^2 &\le \frac{1}{\delta} \sum_i \int_{x_i-\delta/2}^{x_i+\delta/2} |f^{(d-1)}(t)|^2 dt + \sum_i \int_{x_i-\delta/2}^{x_i+\delta/2} |f^{(d-1)}(t) f^{(d)}(t)| dt \\
& \le  \frac{1}{\delta} \int_{T} |f^{(d-1)}(t)|^2 dt + \int_{T} |f^{(d-1)}(t) f^{(d)}(t)| dt.
\end{align*}
The claim then follows from Cauchy-Schwarz and Assumption \eqref{eqn:f'f}.
\end{proof}

\subsection{Number of roots over small intervals} The following also plays a key ingredient in our proofs.
\begin{proposition}[Many roots over small intervals]\label{prop:manyroots} Let $A_0>0$ be a given constant. Let $C^\ast, N$ be given parameters, where $C^\ast = C^\ast_N$ might depend on $N$. Assume that $f$ is a smooth real valued function over $T=[0,1]$ where $f$ has at most $N^{A_0}$ roots, and for any $1\le d  \le A_0 \log N$ we have
\begin{equation}\label{eqn:MB:gen}
\int_T (f^{(d)}(x))^2 dx   \le (C^\ast N)^{2d}  \mbox{ and }  \int_{T} f(x)^2 dx \le (C^\ast)^2.
\end{equation}

Assume that $R, \delta, \eps, \lambda$ are given parameters such that 
$$R\ge 4,\quad \quad  \delta=O\left(\frac{\eps } {\log (1/\eps)}\right), \quad \quad \lambda = \delta^{O(1)}.$$

Assume that there are $\delta N$ disjoint intervals of $T$ of length $R/N$ over which there are at least $\eps N /2$ roots, then there exist a measurable set $A \subset T$ of measure at least $\Theta(\eps /C^\ast)$ over which 
$$\max_{x\in A}|f(x)| \le \lambda  (C^\ast)^{1/2}\quad \mbox{ and } \quad \max_{x\in A} |f'(x)| \le \lambda (C^\ast)^{3/2}  N.$$ 
\end{proposition}

In applications, $N$ is usually $n$ or $n^{1/2+o(1)}$, where $n$ is the degree of our polynomials. As we will be dealing with polynomials of degree $n$, we can take $A_0$ to be $2+o(1)$ in all cases. We will allow $R,\delta, \eps$ to be in the range $[n^{-O(1)},1]$, while $C^\ast$ to be in $[1, \log^{O(1)}n]$. 

We will term the property \eqref{eqn:MB:gen} in Proposition \ref{prop:manyroots} above as Markov-Bernstein property. Motivated by Proposition \ref{prop:manyroots},  we will call a function $f$ {\it exceptional} (with given parameters $R, \delta, \eps, \al, \beta$) if there exists a measurable set $A \subset T$ of measure at least $\Theta(\eps/C^\ast)$ over which 
$$\max_{x\in A}|f(x)| \le \al  (C^\ast)^{1/2} \mbox{ and } \max_{x\in A} |f'(x)| \le \beta (C^\ast)^{3/2} N.$$

To prove the above result we will first need an elementary interpolation result (see for instance \cite[Section 1.1, E.7]{BEbook}).
\begin{lemma}\label{lemma:expansion}
Assume that a $C^\infty$ function $f$ has at least $m$ zeros (counting multiplicities) in an interval $I$ of length $r$. Then 
$$\max_{\theta \in I} |f(\theta)| \le (\frac{4e r}{m})^m \max_{x\in I} |f^{(m)}(x)|$$
as well as 
$$\max_{\theta \in I} |f'(\theta)| \le (\frac{4e r}{m-1})^{m-1} \max_{x\in I} |f^{(m)}(x)|.$$
Consequently, if $f$ has at least $m$ roots on an interval $I$ with length smaller than $(1/8Ce) m/N$, then for any interval $I'$ of length $(1/8Ce) m/N$ and $I \subset I'$ we have 
\begin{equation}\label{eqn:expand:1}
\max_{\theta \in I'} |f(\theta)| \le (\frac{1}{2C})^m (\frac{1}{N})^m \max_{x\in I'} |f^{(m)}(x)|
\end{equation}
as well as 
\begin{equation}\label{eqn:expand:2}
\max_{\theta \in I'} |f'(\theta)| \le n \times (\frac{1}{2C})^{m-1} (\frac{1}{N})^m  \max_{x\in I'} |f^{(m)}(x)|.
\end{equation}
\end{lemma}

\begin{proof} It suffices to show the estimates for $f$ because $f'$ has at least $m-1$ roots in $I$. For $f$, by Hermite interpolation using the roots $x_i$ we have that for any $\theta \in I$ there exists $x\in I$ so that
$$|f(\theta)| = |\frac{f^{(m)}(x)}{m!} \prod_i (\theta-x_i)| \le  \max_{x\in I} |f^{(m)}(x)| \frac{r^m}{m!}.$$
\end{proof}

\begin{proof}(of Proposition \ref{prop:manyroots}) Among the $\delta N$ intervals we first throw away those of less than $\eps \delta^{-1}/4$ roots, hence there are at least $\eps N/4$ roots left. For convenience we denote the remaining intervals by $J_1,\dots, J_M$, where $M \le \delta N$,  and let $m_1,\dots, m_M$ denote the number of roots over each of these intervals respectively.

In the next step (which is geared towards the use of \eqref{eqn:expand:1} and \eqref{eqn:expand:2} of Lemma \ref{lemma:expansion}),
we  expand the intervals $J_j$ to larger intervals $\bar J_j$ 
(considered as union of consecutive closed intervals $J_j$) of length $\lceil c m_j/R \rceil \times (R/N)$ for some small $c$, such as 
$$c=1/(16C^\ast e).$$ 
Furthermore, if the expanded intervals $\bar J_{i_1}',\dots, \bar J_{i_k}'$ of $\bar J_{i_1},\dots, \bar J_{i_k}$ form an intersecting chain, then we create a longer interval
 $\bar J'$ of length $\lceil c(m_{i_1}+\dots+m_{i_k})/R \rceil \times (R/N)$, which contains them and therefore contains at least $m_{i_1}+\dots+m_{i_k}$ roots. After the merging process,
 we obtain a collection 
 $\bar J_1',\dots, \bar J_{M'}'$ with the number of roots $m_1',\dots, m_{M'}'$ respectively, so that $\sum m_i'\geq \eps N/2$.
  Note that now $\bar J_i'$ has length $\lceil cm_i'/R \rceil \times (R/N) \approx cm_i'/N$ (because $\eps \delta^{-1}$ is sufficiently large compared to $R$) and the intervals are $R/N$-separated.

Next, consider the sequence $d_l:=2^l \eps \delta^{-1}/4, l\ge 0$. We classify the sequence $\{m_i'\}$ into collections $G_l$ where 
$$d_l \le m_i' < d_{l+1}.$$ 
Assume that each collection $G_l$ has $k_l=|G_l|$ distinct extended intervals. As each of these intervals has between $d_l$ and $d_{l+1}$ roots, we have 
$$\sum_l k_l d_l \ge \sum_i m_i'/2 \ge \eps N/8.$$
For given $\al, \beta$, we call an index $l$ {\it bad} if 
$$(1/2)^{d_{l}}  (N/2k_{l})^{1/2}  \ge  \la= \min\{\al/8,\beta/8\}.$$ 
That is when
$$k_l \le \frac{N}{2\la^2 4^{d_{l}}}.$$
The total number of roots over the intervals corresponding to bad indices can be bounded by 
$$\sum_i m_i' \le \sum_l k_l d_{l+1}   \le \frac{N}{2\la^2} \sum_{l=0}^\infty\frac{2 d_{l}}{ 4^{d_{l}}} \le \frac{N}{\la^2 2^{\eps\delta^{-1}}} \asymp \frac{N}{\delta^{O(1)} 2^{\eps\delta^{-1}}} \le \eps N/32$$
where we used the fact that $\delta \le  \frac{c_0 \eps } {\log (1/\eps)}$ for some small constant $c_0$.

Now consider a group $G_{l}$ of each good index $l$. Notice that these intervals have length approximately between $cd_l/N$ and $2cd_l/N$. Let $I$ be an interval among the $k_l$ intervals in $G_l$. By the definition of $I$ and by Lemma \ref{lemma:expansion} we have 
\begin{equation}\label{eqn:PP_d}
\max_{x \in I} |f(x)| \le (\frac{1}{2C^\ast})^{d_l} (\frac{1}{N})^{d_l} \max_{x\in I} |f^{(d_l)}(x)| \le \frac{\la}{ (N/2k_{l})^{1/2}  } (\frac{1}{C^\ast N})^{d_l} \max_{x\in I} |f^{(d_l)}(x)|
\end{equation}
as well as 
\begin{equation}\label{eqn:PP_d'}
\max_{x \in I} |f'(x)| \le N \times (\frac{1}{2C^\ast })^{d_l-1} (\frac{1}{N})^{d_l}  \max_{x\in I} |f^{(d_l)}(x)|  \le N \times \frac{2\la C^\ast}{ (N/2k_{l})^{1/2}  } (\frac{1}{C^\ast N})^{d_l} \max_{x\in I} |f^{(d_l)}(x)|.
\end{equation}
On the other hand, as these $k_l$ intervals are $R/N$-separated (and hence $4/N$-separated), by the assumption of our proposition and by Lemma \ref{lemma:largesieve}  we have 
$$\sum_{\bar J_i'\in G_l}\max_{x\in \bar J_{i}'}(f^{(d_{l})}(x))^2 \le 4 (C^\ast N)^{2d_l+1}.$$
Hence we see that for at least half of the intervals $J_i'$ in $G_l$ 
$$\max_{x\in J_{i}'}|f^{(d_{l})}(x)| \le 4 (C^\ast N/k_{l})^{1/2} (C^\ast N)^{d_l} .$$ 
It follows from \eqref{eqn:PP_d} and \eqref{eqn:PP_d'} that over these intervals
$$\max_{x \in J_{i}'} |f(x)| \le  \frac{\la}{ (N/2k_{l})^{1/2}  } (\frac{1}{C^\ast N})^{d_l}  4 (C^\ast N/k_{l})^{1/2} (C^\ast N)^{d_{l}}  \le  4\la (C^\ast)^{1/2}$$
and similarly,
$$\max_{x \in J_i'} |f'(x)| \le N \times \frac{2 \la C^\ast}{ (N/2k_{l})^{1/2}  } (\frac{1}{C^\ast N})^{d_l} 4 (C^\ast N/k_{l})^{1/2} (C^\ast N)^{d_{l}}   \le 8 \la (C^\ast)^{3/2} N.$$

Letting $A_l$ denote the union of all such intervals $J_{i}'$ of a given good index $l$, and letting $A$ denote the union of the $A_l$'s over all  good indices $l$, we obtain (with $\mu$ denoting Lebesgue measure) 
\begin{align*}
\mu(A) \ge \sum_{l, \textsf{good}} (cd_l/N) k_l/2& \ge \sum_{l, \textsf{good}} (c/4) d_{l+1}k_l/N \ge \sum_{l, \textsf{good}} (c/4) m_l'/N \\
& \ge (c/4)(\eps N/8 -\eps N/32)/N \ge \frac{\eps}{1024 C^\ast e}. 
\end{align*}
Finally, notice that over $A$ we have $\max_{x \in A} |f(x)| \le 4\la (C^\ast)^{1/2}$ and $\max_{x \in A} |f'(x)|\le 8 \la (C^\ast)^{3/2}  N$. 
\end{proof}


\subsection{Stability}

We next give a deterministic result (see also \cite[Claim 4.2]{NS}) to control the number of roots under perturbation.

\begin{lemma}\label{lemma:stab} Fix strictly 
positive numbers $\mu$ and $\nu$. Let $I=(a,b)$ be an  interval of length greater than $2\mu/\nu$, and let $f$ be a  $C^1$-function on $I$ such that at each point $x\in I$ we have either $|f(x)|> \mu$ or $|f'(x)| > \nu$. Then for each root $x_i \in I$ with $x_i-a>\mu/\nu$ and $b-x_i>\mu/\nu$ there exists an interval $I(x_i) = (a',b')$ where $f(a')f(b')<0$ and $|f(a')|=|f(b')|=\mu$,
such that $x_i \in I(x_i) \subset (x_i-\mu/\nu, x_i + \mu/\nu)$ and the intervals $I(x_i)$ over the roots are disjoint.
\end{lemma}

\begin{proof}(of Lemma \ref{lemma:stab}) We may and will assume that $f$ is not constant on $I$.
By changing $f(x)$ to $\la_1 f(\la_2 x)$ for appropriate $\lambda_1,\lambda_2$, it suffices to consider $\mu=\nu=1$.
For each root $x_i$, and for $0< t \le 1$ consider the interval $I_t(x_0)$ containing $x_0$ of those points $x$ where $|f(x)| < t$. We first show that for any $0<t_1,t_2\le 1$ we have that $I_{t_1}(x_1)$ and $I_{t_2}(x_2)$ are disjoint for distinct roots $x_i \in I$ satisfying the lemma's assumption. Assume otherwise, because $f(x_1)=f(x_2)=0$, there exists $x_1<x<x_2$ such that $f'(x)=0$ and $|f(x)| \le \min\{t_1,t_2\}$, and so contradicts with our assumption. We will also show that $I_1(x_0)\subset (x_0-1,x_0+1)$. Indeed, assume otherwise for instance that $x_0-1\in I_1(x_0)$, then for all $x_0-1<x<x_0$ we have $|f(x)|<1$, and so $|f'(x)|>1$ over this interval. Without loss of generality, 
we assume $f'(x)>1$ for all $x$ over this interval. The mean value theorem would then imply that $|f(x_0-1)|= |f(x_0-1)-f(x_0)|>1$, a contradiction with $x_0-1 \in I_1(x_0)$.  As a consequence, we can define $I(x_i)=I_1(x_i)$, for which at the endpoints the function behaves as desired.  
\end{proof}

\begin{corollary}\label{cor:stab} Fix positive $\mu$ and $\nu$. Let $I=(a,b)$ be an  interval of length at least $2 \mu/\nu$, and let $f$ be a $C^1$-function on $I$  such that at each point $x\in I$ we have either $|f(x)|> \mu$ or $|f'(x)| > \nu$. Let $g$ be a function such that $|g(x)|< \mu$ over $I$. Then for each root $x_i \in I$ of $f$ with $x_i-a>\mu/\nu$ and $b-x_i>\mu/\nu$ we can find a root $x_i'$ of $f+g$ such that $x_i' \in (x_i-\mu/\nu, x_i + \mu/\nu)$, and also the $x_i'$ are distinct.
\end{corollary}

Our next two ingredients are of opposite nature, one is one concentration, and the other is on anti-concentration.

\subsection{Repulsion}\label{subsection:repulsion:orth} 

Our next key ingredient is a {\it repulsion-type} estimate which shows that at any point it is unlikely that the function and its derivative vanish simultaneously.  We will deduce our repulsion result via the study of small ball probability of the random walk in $\R^2$ 
$$(F(x), \frac{1}{N} F'(x)) =: \sum_{i=0}^n \xi_i \Bu_i,$$
Here again $N$ will be either $n$ or $n^{1/2-o(1)}$ in later applications depending on the model (i.e. on the growth of $|F'|$). We also remark that $F$ here can either take the form of \eqref{eqn:f:general} (for the orthorgonal and Elliptic polynomials), or a normalized version of it for Weyl polynomials.

\begin{lemma}\label{lemma:smallball} Assume that for all $(a_1,a_2)\in \BS^1$ we have 
\begin{equation}\label{eqn:non-degenerating}
\sum_{i=1}^n  \langle  \Bu_i(x), (a_1,a_2) \rangle^2 \asymp 1.
\end{equation}
Then for for any $r \ge 1/\sqrt{n}$ we have 
$$\sup_{a\in \R^2}\P\left (\sum_i \xi_i  \Bu_i  \in B(a,r)\right )= O(r^2).$$
\end{lemma} 
The property $\sum_{i=1}^n  \langle  \Bu_i(x), (a_1,a_2) \rangle^2 \asymp 1$ above is called {\it non-degenerating}.
\begin{proof} This is \cite[Theorem 1]{Halasz} where we cover a ball of radius $r$ by $nr^2$ balls of radius $1/\sqrt{n}$. 
\end{proof}
We also refer the reader to   \cite{NgV-survey,RV-rec} for further developments of similar anti-concentration estimates. We deduce from Lemma \ref{lemma:smallball} the following corollary.
\begin{theorem}[Repulsion estimate]\label{thm:repulsion} Assume that $\xi$ has mean zero and variance one, and $\Bu_i(x)$ are as in Lemma \ref{lemma:smallball}. Then as long as $\al> 1/\sqrt{n}$,  $\beta>1/\sqrt{n}$ we have
$$\P\big( |F(x)|\le \al \wedge \frac{1}{N}|F'(x)| \le \beta \big) = O(\al \beta).$$
\end{theorem}
\begin{proof} The event under consideration can be written as 
$$(F(x), \frac{1}{N} F'(x)) = \sum_i \xi_i \Bu_i \in [-\al, \al] \times [-\beta, \beta].$$
\end{proof}
 
In application we just choose $\al,\beta$ to be at least $n^{-c}$ for some small constant $c$. 

 \subsection{Geometric concentration}
On the probability side, for bounded random variables we will rely on the following consequence of McDiarmid's inequality. 


For random variables $\xi$ satisfying the  log-Sobolev inequality, that is so that there is a positive constant $C_0$ such that for any smooth, bounded, compactly supported functions $f$ we have 
\begin{equation}\label{eqn:logSobolev}
\Ent_\xi(f^2) \le C_0 \E_\xi|\nabla f|^2,
\end{equation}
where $\Ent_\xi(f) = \E_\xi(f \log f) - (\E_\xi(f)) (\log \E_\xi(f))$, we  use the following.

\begin{theorem}\label{thm:sobolev}  Assume that $\bxi=(\xi_1,\dots, \xi_n)$, where $\xi_i$ are iid copies of $\xi$ satisfying \eqref{eqn:logSobolev} with a given $C_0$. 

\begin{itemize}

\item Let $F: \Omega^n \to \R$ be a function of Lipschitz constant 1 with respect to the Euclidean distance $d_2(\cdot)$. Then  
$$\P(F(\bxi) \ge m(F(\bxi) + s) )\le 2\exp(-s^2/4C_0)$$
where $m(\cdot)$ is the median. 
\vskip .1in
\item Let $\CA$ be a set in $\R^n$. Then for any $s>0$ we have
$$\P\big(d_2(\bxi,\CA) \ge s  \big) \le 2 \exp\big(-\P^2(\Bx\in \CA) s^2/4C_0 \big).$$
In particularly, if $\P(\bxi\in \bar{\CA})\ge 1/2$ then $\P(d_2(\bxi,\bar{\CA}) \ge s ) \le 2 \exp(-s^2/16C_0)$. Similarly if $\P(d_2(\bxi,\CA) \ge s)\ge 1/2$ then $ \P(\bxi\in \CA)\le 2 \exp(-s^2/16C_0)$.
\end{itemize}
\end{theorem}

To make it consistent 
with Theorem \ref{thm:bounded}, sometimes we will also use $\Omega^n$ in place of $\R^n$ in various applications of Theorem \ref{thm:sobolev}. For bounded random variables, we use the following McDiarmid’s inequality.

 \begin{theorem}\label{thm:bounded} Assume that $\bxi=(\xi_1,\dots, \xi_n)$, where $\xi_i$ are iid copies of $\xi$ of mean zero, variance one, taking values in $\Omega$, a subset of $[-C_0,C_0]$. Let $\CA$ be a set of $\Omega^n$. Then for any $s>0$ we have
$$\P(\bxi \in \CA) \P(d_2(\bxi, \CA) \ge s) \le 4\exp(-s^4 /16C_0^4 n).$$
\end{theorem}
In particular, 
we see that this result is meaningful only if $s \gg n^{1/4}$. These results are standard, whose proof can be found for instance in \cite[Appendix B]{NgZ}. We also refer the reader to  the classical source \cite{L} for further results.

While the above will be useful for the proofs of Theorem \ref{thm:orthogonal} and Theorem \ref{thm:Weyl}, for Theorem \ref{thm:Elliptic} we will use the following corollary in lower dimension.

\begin{theorem}\label{thm:E:concentration}   Assume that $N \le n$ and let $\Bv_1=(v_{11},\dots, v_{1n}),\dots, \Bv_N=(v_{N1},\dots, v_{Nn})$ be given $N$ orthogonal unit vectors in $\Omega^n$. Let $\CA \subset \Omega^n$ be such that 
$$\P(\bxi \in \CA) \le 1/2.$$ 
Let $s>0$ and assume that $\CE \subset \CA$ such that for any $\Bf\in \CE$ and any $\Bg=(g_1,\dots, g_n)\in \R^n$ for which $\sum_i (\Bg \cdot \Bv_i)^2 \le s^2$ we have  
\begin{equation}\label{eqn:g:E}
\Bf+\Bg \in \CA.
\end{equation}
\begin{itemize}
\item  Assume that $\Bx=(x_1,\dots, x_n)$, where $x_i$ are iid copies of $\xi$ satisfying \eqref{eqn:logSobolev} with a given $C_0$.
Then 
$$\P(\bxi \in \CE) \le 2\exp(-s^2/4C_0).$$
\item  Assume that $\Bx=(x_1,\dots, x_n)$, where $x_i$ are iid copies of $\xi$ of mean zero, variance one, taking values in $\Omega$, a subset of $[-C_0,C_0]$. Then 
$$\P(\bxi \in \CE) \le 4 \exp(-s^4/16n C_0^4) .$$

\end{itemize}
\end{theorem}

\begin{proof}(of Theorem \ref{thm:E:concentration}) Consider the distance via projection onto the subspace generated by $\Bv_1,\dots, \Bv_n$,
$$d_{2,\Bv}(\Ba,\Bb):= \sum_i ((\Ba-\Bb)\cdot \Bv_i)^2$$ 
and 
$$d_{2, \Bv}(\Ba,\CS):= \inf_{\Bs\in S}d_{2,\Bv}(\Ba,\Bs).$$
By the triangle inequality and by the fact that $\Bv_1,\dots, \Bv_N$ are orthogonal unit vectors, we see that $d_{2,\Bv}(\Bx, S)$ is 1-Lipschitz with respect to the Euclidean distance $d_2(\cdot)$. 
This is because $d_{2,\Bv}(\Bx, S) -d_{2,\Bv}(\By,S) \le d_{2,\Bv}(\Bx, \Bs_\Bx) -d_{2,\Bv}(\By,\Bs_\Bx)$ for some $s_\Bx\in S$, and $d_{2,\Bv}(\Bx, \Bs_\Bx) -d_{2,\Bv}(\By,\Bs_\Bx) \le d_{2,\Bv}(\Bx, \By) \le d_2(\Bx,\By)$. 

Let 
$$F(\bxi) := d_{2, \Bv}(\bxi, \bar{\CA}).$$
By definition $F(\Bx) \ge s$ for any $\Bx \in \CE$, and as $\P(\bar{\CA}) \ge 1/2$, the median of $F(\bxi)$ is zero. It is known from Theorem \ref{thm:sobolev} that for distributions satisfying log-Sobolev inequality 
$$\P(F(\bxi) \ge m(F(\bxi)) + s )\le 2\exp(-s^2/4C_0).$$
Hence 
$$\P(\bxi \in \CE) \le \P(F(\bxi) \ge s) \le2 \exp (-\tau^2 N/4C_0).$$

Now for the boundedness case, note that if we change the $j$-th coordinate of $\Ba=(a_1,\dots, a_j,\dots, a_n)$ to $\Ba'=(a_1,\dots, a_j',\dots, a_n)$, then
$$ |d_{2,\Bv}^2(\Ba,\Bb) -  d_{2,\Bv}^2(\Ba',\Bb)| \le  (a_j - a_j')^2  \sum_{i=1}^N (\Be_j \cdot \Bv_i)^2 \le 4C_0^2.$$
Hence the function $G(\Bx) = d_{2,\Bv}^2(\Bx, \bar{\CA})$ is $4C_0^2$-Lipschitz (coordinatewise). By  McDiarmid’s inequality, with $\mu = \E F(\bxi)$, for any $\la>0$
$$\P(|G(\bxi)- \mu| \ge \la) \le 2 \exp(-\la^2/8n C_0^4).$$
We can then deduce from here that $\P(G(\bxi)=0) \P(G(\bxi) \ge \la ) \le 4 \exp(-\la^2/16n C_0^4)$, and hence 
$$\P(\bxi \in \CE) \le  \P(G(\bxi) \ge s^2) \le  4 \exp(-s^4/16n C_0^4).$$
\end{proof}

\subsection{Proof methods} 
We next highlight the proof method for Theorem \ref{thm:orthogonal}, Theorem \ref{thm:Elliptic}, and Theorem \ref{thm:Weyl}. Broadly speaking, our overall approach follows \cite{NgZ}, which in turn was motivated by the perturbation framework of Nazarov and Sodin \cite{NS} in their proof of Theorem \ref{thm:NS}. However, the situation for univariate polynomials is totally different compared to spherical harmonics (or arithmetic random waves as in \cite{Rozen}), for instance \cite[Claim 2.2]{NS} and \cite[Claim 2.4]{NS} do not hold for roots of polynomials at all. The detailed plan follows below.

\begin{enumerate}
\item Our starting points are  Theorem \ref{thm:E:gen}, Theorem \ref{thm:W:gen}, and Theorem \ref{thm:O:gen} which imply that $N_T(F)$ are moderately concentrated around their means in all models. These results are used in both treatments of the upper tail and lower tail. Roughly speaking, these results guarantee that certain events will have probability at least 1/2, an important input in the applications of geometric concentration.
\vskip .1in
\item We will show that it is highly unlikely that there is a large set of intervals where both $|F|$ and $|F'|$ are both small (i.e. ``exceptional" polynomials are very rare). We justify this by using geometric concentration together with Lemma \ref{lemma:largesieve} and Theorem \ref{thm:repulsion}.
\vskip .1in
\item For $F$ that is not exceptional, we will use Proposition \ref{prop:manyroots} to conclude that the number of roots over the intervals where $|F|$ and $|F'|$ are small simultaneously is small. 
\vskip .1in
\item Basing on the above results, and on the stability result Corollary \ref{cor:stab} we have that if $N_T(F)$ is close to $\E N_T(F)$ (with high probability, from the first step), then $N_T (F+g)$ is also close to $N_T(F)$ as long as $\|g\|_2$ is small. As such, geometric tools such as Theorem \ref{thm:sobolev}, Theorem \ref{thm:bounded},  and Theorem \ref{thm:E:concentration} can be invoked again to show that indeed $N_T(F)$ satisfies exponential concentration as desired. 
\end{enumerate}
While in the random orthogonal polynomials setting (Theorem \ref{thm:orthogonal}) the above perturbation framework is rather clear -- that we are allowed to perturb (the coefficients of) $F$ by $g$ of order $\sqrt{n}$ in $L_2$-norm (see \eqref{eqn:O:g}), for Theorem \ref{thm:Weyl} of the Weyl model  $g$ will be designed to have much smaller norm, see \eqref{MB:W:g} and \eqref{eqn:gb}. Especially, for the Elliptic model we will perturb groups of coefficients rather than individual ones (see \eqref{eqn:g:E}), and one of our findings is a canonnical way to group these coefficients. Other highlights  might include repulsion bound for random walks with steps from orthogonal polynomials, Corollary \ref{cor:repulsion:O} (which is based on beautiful properties of orthogonal polynomials from \cite{LP} and the references therein) or our bound on the higher derivatives of Weyl polynomials (Lemma \ref{lemma:Weyl:derivative}), but one of the trickiest parts of the current paper is to find appropriate variants of Markov-Bernstein type estimates for Elliptic and Weyl polynomials, throughout \eqref{eqn:E:Bern} and \eqref{eqn:W_0} respectively, to serve as inputs for Proposition \ref{prop:manyroots} as well as Lemma \ref{lemma:largesieve}.


\section{Orthogonal polynomials: proof of Theorem \ref{thm:orthogonal}}\label{Section:orthogonal}
We note that in this section $a,a',b,b'$ are fixed and $n \to \infty$. 
\subsection{Properties of orthogonal polynomials}\label{subsection:supporting:orth}
Recall the orthogonal polynomials $p_j(x)$ from Theorem \ref{thm:O:gau}. Using the same assumption as in that result, for short set 
$$T=[a,b]\quad \mbox{ and } \quad T'=[a',b'],$$ 
where we assume that 
$$a'=a - \eps_0 (b-a),\quad \quad \quad  b' = b+\eps_0 (b-a).$$ 
Consider the following (deterministic) polynomials
\begin{equation}\label{eqn:F:O'}
f(x)=\frac{1}{\sqrt{n}}\sum_{j=1}^n a_j p_j(x).
\end{equation}
First of all, recall that for any polynomial $f$ of degree $n$, Bernstein's inequality says that
$$|f'(x)| \le \frac{n}{\sqrt{1-x^2}} \|f\|_{L_\infty[-1,1]}, x\in (-1,1)$$
and for some absolute constant $C$
$$\int_{-1}^1|\sqrt{1-x^2} f'(x)|^2 dx \le Cn^2 \int_{-1}^1 f(x)^2 dx.$$
Hence
\begin{lemma}\label{lemma:Bern:orth} Let $f$ be a polynomial of degree $n$. For all $x\in T=[a,b]$
$$|f'(x)| \le  (\eps_0 (b'-a'))^{-1} Cn \|f\|_{L_\infty[a',b']}$$
and
$$\int_{a}^b f'(x)^2 dx \le (\eps_0 (b'-a')^2)^{-1} C n^2 \int_{a'}^{b'} f(x)^2 dx$$
and
$$\int_{a}^b f''(x)^2 dx \le  (\eps_0 (b'-a')^2)^{-2} C^2 n^4\int_{a'}^{b'} f'(x)^2 dx$$
and more generally, for all $k\le n$,
$$\int_{a}^b (f^{(k)}(x))^2 dx \le (\eps_0 (b'-a')^2)^{-k} C^k n^{2k} \int_{a'}^{b'} f(x)^2 dx.$$
\end{lemma}

\begin{proof} Let $g(x) = f(  x(b'-a')/2 + (a'+b')/2)$, then $g$ is also a polynomial of degree $n$. By Bernstein's inequality, 
$$\int_{-1+\eps_0}^{1-\eps_0} g'(x)^2 dx \le C \eps_0^{-1} n^2 \int_{-1}^{1} g(x)^2 dx.$$ 
Passing back to $f$, 
$$ \int_{a}^b f'(x)^2 dx \le C \eps_0^{-1} ((b'-a')^2/4)^{-1} n^2 \int_{a'}^{b'} f(x)^2 dx.$$
For general $k$, for each $1\le i\le k$ we let $I_i =(a_i',b_i')=(a' +  (i/k)\eps_0 (b'-a'), b'- (i/k)\eps_0 (b'-a'))$ 
and iterate the above bound to $f^{(i)}$ and $f^{(i-1)}$ over the intervals $I_i$ and $I_{i-1}$ respectively, we obtain (with room to spare)
$$\int_{ a' + \eps_0 (b'-a') }^{b'-\eps_0 (b'-a')  } [f^{(k)}(x)]^2 dx \le C^k  \sqrt{\frac{k^k}{k!}} \eps_0^{-k} ((b'-a')^2/4)^{-k} n^{2k} \int_{a'}^{b'} f(x)^2 dx \le {C'}^k  (\eps_0 (b'-a')^2)^{-k}n^{2k}\int_{a'}^{b'} f(x)^2 dx,$$
where $C'$ is another absolute constant.
\end{proof}

Now for the polynomials $p_j(x)$, we gather here gather a  few facts from \cite[Lemma 3.2, Lemma 3.3]{LP}. 
\begin{lemma}\label{lemma:p_i:orth} There exists a constant $C$ (depending $a,b$ and the density $\omega$ over $[a,b]$) such that the following estimates hold for $x_0\in T$
\begin{itemize}
\item
$$\sum_{i=1}^n p_i(x_0)^2 \ge n/C \mbox{ and } \sum_{i=1}^n p'_i(x_0)^2 \ge n^3/C ;$$
\item 
$$|p_j^{(d)}(x_0)| = O(n^{d}), d=0,1,2;$$
\item  for any $x\in [a,b]$
$$\lim_{n\to \infty}\frac{1}{n^2} \sum_{j=0}^n p_j(x)p_j'(x) =0.$$
\end{itemize}
\end{lemma}
As a consequence we deduce the following non-degenerating property.
\begin{claim}\label{claim:iso:orth} Let $\Bu_i(x) =\frac{1}{\sqrt{n}}(p_i(x), p_i'(x)/n)$. For all $(a_1,a_2)\in \BS^1$ we have 
$$\sum_i  \langle  \Bu_i, (a_1,a_2) \rangle^2 \asymp 1.$$
\end{claim}
\begin{proof} We have
\begin{align*}
\sum_i \langle \Bu_i, (a_1,a_2) \rangle^2 &=  \frac{1}{n}a_1^2 \sum_{i} p_i(x)^2  + \frac{1}{n} a_2^2 \sum_i \frac{1}{n^2}p_i'(x)^2  + 2 \frac{1}{n}a_1a_2 \sum_i p_i(x) \frac{1}{n}p_i'(x) \asymp 1,
\end{align*}
where we used the fact from Lemma \ref{lemma:p_i:orth}. Notice also that $\|\Bu_i\|_2 \ll 1$. The above claim implies that a positive portion of the $\{ |\langle \Bu_i, (a_1,a_2) \rangle| \}$ are in fact  of order 1. 
\end{proof}
As an immediate consequence of this result, we deduce the following corollary of Theorem \ref{thm:repulsion} with $N=n$.

\begin{corollary}\label{cor:repulsion:O} Assume that $F$ has the form of \eqref{eqn:F:O'} where the coefficients are iid copies of $\xi$ of mean zero and variance one. Then as long as $\al> 1/\sqrt{n}$,  $\beta>1/\sqrt{n}$ we have
$$\P\big( |F(x)|\le \al \wedge \frac{1}{n}|F'(x)| \le \beta \big) = O(\al \beta).$$
\end{corollary}

\subsection{Exceptional polynomials}\label{subsection:exceptional:O} This current section is motivated by the treatment in \cite[Section 4.2]{NS} and \cite[Section 4]{NgZ}. Recall $T, T'$ from the previous subsection. Consider a polynomial of type 
\begin{equation}\label{eqn:F:O}
F(x) = \frac{1}{\sqrt{n}} \sum_{i=0}^n \xi_i p_i(x).
\end{equation}
For convenience we will identify $F$ with the vector $\Bv_{F}=(\xi_0,\dots, \xi_n) \in \Omega^{n+1}$, which is a random vector when the $a_i$ are random.

Next, in both the boundedness and the log-Sobolev cases, the event $\CE_{b}$ that 
$$\frac{\sum_i a_i^2}{n} = O(1)$$ 
has probability at least $1 -\exp(-\Theta(n))$, in which case by the orthogonality (and by Lemma \ref{lemma:p_i:orth}) we have
\begin{equation}\label{eqn:M_f}
 \int_T (F(x))^2 dx  \le  \int_{T'} (F(x))^2 dx  \le C^\ast,  \mbox{ where $C^\ast$ is a sufficiently large constant}.
\end{equation} 
Hence for this section, without loss of generality we are working on this event $\CE_b$.

Let $R=C^\ast$ and cover $T$ by $n/R$ open interval $I_i$ of length $R /n$ each. Let $3 I_i$ be the interval of length $3R/n$ having the same midpoint with $I_i$. Given some parameters $\al, \beta$, we call an interval $I_i$ {\it stable} for a function $f$ (of form \eqref{eqn:F:O}) if there is no point in $x\in 3I_i$ such that $|f(x)|\le \al$ and $|f'(x)|\le \beta n (b-a)$. In other words, there is no $x\in 3I_i$ where $|f(x)|$ and $|f'(x)|/n$ are both small. Let $\delta$ be another small parameter (so that $\delta R <1/4$), we call $f$ {\it exceptional} if the number of unstable intervals  is at least $\delta n$. We call $f$ not exceptional otherwise. 

Let $\CE_e = \CE_e(R,\al,\beta; \delta)$ denote the set of vectors $\Bv_{f}$ associated to exceptional function $f$ of form \eqref{eqn:F:O}. 
Our goal in this section is the following.

\begin{theorem}\label{thm:exceptional} Let $F$ be a random polynomial of form \eqref{eqn:F:O}, where the $\xi_i$ are as in Theorem \ref{thm:orthogonal}.  Assume that $\al,\beta,\delta$ satisfy $\delta \le (C^\ast)^{-8}$ and
\begin{equation}\label{eqn:parametersthm:O}
\al \asymp \delta^{3}, \beta \asymp \delta^{3/2},  \delta > n^{-1/4}.
\end{equation}
Then we have
$$\P\Big(\Bv_{F} \in \CE_e\Big) \le 4 e^{-\delta^{16} n/16 C_0^4}.$$
\end{theorem}

To justify this result, first assume that $f$ (playing the role of $F$) is exceptional, then there are $K=\lfloor \delta n/3 \rfloor$ unstable intervals that are $R/n$-separated.

 Now for each unstable interval in this separated family we choose $x_j \in 3 I_j$ where  $|f(x_j)|\le \al$ and $|f'(x_j)|\le \beta (b-a)n$ and consider the interval $B(x_j, \gamma(b-a)/n)$ for some $\gamma <1$ chosen sufficiently small (given $\delta$, see \eqref{eqn:parameters:O}). 
 Let 
$$M_j:= \max_{x\in B(x_j,\gamma/N)} |f''(x)|.$$  
By Lemma \ref{lemma:largesieve} (with $N=n$, and the conditions are satisfied by Lemma \ref{lemma:Bern:orth} and \eqref{eqn:M_f}) we have
$$\sum_{j=1}^K M_j^2 \le (C^\ast n) \int_{x \in T} f''(x)^2 dx \le  (C^\ast n)^5  \int_{x \in T'} f(x)^2 dx \le (C^\ast)^6  n^5   .$$
We thus infer from the above that the number of $j$ for which $M_j \ge C_1 \delta^{-1/2}(C^\ast)^3 n^2$ is at most $C_1^{-2} \delta n$. Hence for at least $(1/3 - C_1^{-2})\delta n$ indices $j$ we must have $M_j <C_1 \delta^{-1/2} (C^\ast)^3 n^2$.

Consider our function over $B(x_j, \gamma/n)$, then by Taylor expansion of order two around $x_j$, we obtain for any $x$ in this interval
$$ |f(x)| \le |f(x_j)| + |f'(x_j)| (\gamma/n) + \max_{t\in B(x_j, \gamma/n)}|f''(t)| (\gamma/n)^2 \le \al + \beta \gamma + C_1 \delta^{-1/2} (C^\ast)^3 \gamma^2/2 $$
and similarly
$$|f'(x)| \le (\beta + C_1 \delta^{-1/2} (C^\ast)^3 \gamma) n.$$
Now consider a function $g$ of the form \eqref{eqn:F:O}
$$g(x)= \frac{1}{\sqrt{n}}\sum_{n} a'_i p_i(x),$$ 
where $a_i'$ are the amount we want to perturb in $f$. We assume that 
\begin{equation}\label{eqn:O:g}
\|g\|_2^2 = \frac{1}{n} \sum_i (a_i')^2 \le \tau^2.
\end{equation}
Then as the intervals $B(x_j, \gamma/n)$ are $R/n$-separated, by Lemma \ref{lemma:largesieve} (and  Lemma \ref{lemma:Bern:orth}) we have
$$\sum_j \max_{x \in B(x_j, \gamma/\la)} g(x)^2  \le ((C^\ast)^2 n ) \frac{\sum_i {a'}_i^2}{n} \le n \tau^2 (C^\ast)^2$$
and 
$$\sum_j \max_{x \in B(x_j, \gamma/\la)} g'(x)^2  \le n^3 \tau^2 (C^\ast)^3.$$
Hence, again by an averaging argument, the number of intervals where either $\max_{x \in B(x_j, \gamma/n)} |g(x)| \ge C_2 \delta^{-1/2} C^\ast \tau$ or   $\max_{x \in B(x_j, \gamma/n)} |g'(x)| \ge C_2 \delta^{-1/2} \tau (C^\ast)^{3/2} n$ is bounded from above by $2 C_2^{-2}\delta n$. On the remaining at least $(1/3 - C_1^{-2} -  2 C_2^{-2})\delta n$ intervals, with $h=f+g$, we have simultaneously that 
$$|h(x)| \le  \al + \beta \gamma + C_1 \delta^{-1} \gamma^2 (C^\ast)^3 + C_2 \delta^{-1/2}  (C^\ast)^2 \tau \mbox{ and } |h'(x)| \le  (\beta + C_1 \delta^{-1} \gamma  (C^\ast)^3 + C_2 \delta^{-1/2} \tau (C^\ast)^{3/2}) n.$$
For short, let 
$$\al':=  \al + \beta \gamma + C_1 \delta^{-1} \gamma^2 (C^\ast)^3 + C_2 \delta^{-1/2}  (C^\ast)^2 \tau  \mbox{ and } \beta':= \beta + C_1 \delta^{-1} \gamma  (C^\ast)^3 + C_2 \delta^{-1/2} \tau (C^\ast)^{3/2} .$$ 
It follows that $\Bv_h$ belongs to the set $\CU=\CU(\al, \beta,\gamma,\delta, \tau, C_1,C_2,C_3)$ in $\Omega^{n+1}$ of the vectors  corresponding to $h$, for which the measure of $x$ with $|h(x)| \le  \al'$ and $|h'(x)| \le  \beta' N $ is at least $2(1/3 - C_1^{-2} -  2 C_2^{-2})\delta \gamma$ (because this set of $x$ contains $(1/3 - C_1^{-2} -  2 C_2^{-2})\delta n$ intervals of length $2\gamma/n$). 
Putting together we have obtained the following claim. 

\begin{claim}\label{O:claim2} Assume that $\Bv_{f} \in \CE_e$. Then for any $g$ with $\|g\|_2 \le \tau$ we have $\Bv_{f+g} \in \CU$. In other words,
$$\Big \{\Bv\in \Omega^{n}, d_2(\CE_e, \Bv) \le \tau \sqrt{n}\Big \} \subset \CU.$$
\end{claim}
We next show that for $f$ of form \eqref{eqn:F:O} with $\xi_i$ are in Theorem \ref{thm:orthogonal}
\begin{equation}\label{eqn:U:O}
\P(\Bv_{f} \in \CU) \le 1/2.
\end{equation}
 Indeed, for each $f$, let $B(f)$ be the measurable set of  $x\in T$  such that $\{|f(x)| \le \al'\} \wedge \{f'(x)| \le  \beta' n\}$. Then the Lebesgue measure of $B(f)$, $\mu(B(f))$, is bounded by 
$$\E \mu(B(f)) = \int_{x \in  T} \P(\{|f(x)| \le \al'\} \wedge \{|f'(x)| \le  \beta' n\}) dx = O(\al' \beta' ),$$
where we used Corollary \ref{cor:repulsion:O} for each $x$.  

It thus follows that $\E \mu(B(f)) = O(\al' \beta')$. So by Markov inequality,
\begin{equation}
\label{eq-1011:O}
\P(\Bv_{f} \in \CU) \le \P\big ( \mu(B(f)) \ge 2(1/3 - C_1^{-2} -  2 C_2^{-2})\delta \gamma \big ) = O(\al' \beta'/\delta \gamma) <1/2
\end{equation}
if $\al, \beta$ are as in \eqref{eqn:parametersthm} and then $\gamma, \tau$ are chosen appropriately (noting that $C^\ast, C_1,C_2$ are just constants), for instance \footnote{Our choices of parameters, here and later, are not optimized. However, our perturbation method is unlikely to yield optimal dependencies among the parameters.} as
\begin{equation}\label{eqn:parameters:O}
\gamma \asymp \delta^{3/2}, \tau \asymp \delta^4. 
\end{equation}

\begin{proof}(of Theorem \ref{thm:exceptional}) By Theorems \ref{thm:sobolev} and \ref{thm:bounded}, and by Claim \ref{O:claim2} and \eqref{eq-1011:O} we have
$$\P(\Bv\in \CE_e) \le 4 e^{-\tau^4 n/16C_0^4}.$$
\end{proof}

\subsection{Roots over unstable intervals}\label{subsection:unstable}

Let $0<\eps<1$ be given.  Let $\delta$ be chosen so that 
\begin{equation}\label{delta:eps:O}
\delta =O(\min\{\eps/\log(1/\eps), \eps/(C^\ast)^2\}).
\end{equation}
Also let $\la=\delta^4$. Hence the conditions in Proposition \ref{prop:manyroots} are satisfied. Let $\al , \beta$ be chosen as in Theorem \ref{thm:exceptional:E}.  We clearly have
$$\la (C^\ast)^{1/2} \le \al,  \la (C^\ast)^{3/2} \le \beta.$$
Our first observation here is that

\begin{cor}\label{cor:manyroots:O} With the parameters as above, a non-exceptional $f$ (defined in Subsection \ref{subsection:exceptional:O} ) cannot have more than $\eps n/2$ roots over any $\delta n$ intervals $I_i$. In particularly, $f$ cannot have more than $\eps n/2$ roots over the unstable intervals. 
\end{cor}
\begin{proof}  Assume that $f$ has more than $\eps n/2$ roots over any $\delta n$ intervals $I_i$. By Proposition \ref{prop:manyroots} (again \eqref{eqn:MB:gen} is satisfied by Lemma \ref{lemma:Bern:orth}, and here without loss of generality our current interval $T=[a,b]$ is $[0,1]$), there exists a set $A$ of measure at least  $\eps/C^\ast$ over which $|f(x)| \le \la (C^\ast)^{1/2} \le \al$ and $|f'(x)| \le (C^\ast)^{3/2} n \le \beta n$. Because the measure of $A$ is at least $\delta R$ (i.e. $\delta C^\ast$) (as $\eps/C^\ast \ge \delta C^\ast$), $A$ must intersect with the stable intervals, but this contradicts with the definition of stable intervals over which $|f(x)|$ and $|f'(x)|/n$ cannot be small at the same time. 
\end{proof}

Consider a random polynomial $F$ of form \eqref{eqn:F:O}. If $F$ is not exceptional we let $S(F)$ be the collection of intervals over which $F$ is stable. Let $N_s(F)$ denote the number of roots of $F$ over the set $S(F)$ of stable intervals.

\begin{cor}\label{cor:control:O} With the same parameters as in  Corollary \ref{cor:manyroots:O}, we have
$$\P\Big(N_s(F) \1_{F \in \CE_e^c} \le \E N_T(F) - \eps n \Big)=o(1)$$
and  
$$\E \Big(N_s(F) \1_{F \in \CE_e^c}\Big) \ge N_T(F)-  2 \eps n/3.$$
\end{cor}
\begin{proof}(of Corollary \ref{cor:control:O}) For the first bound, by Corollary \ref{cor:manyroots:O}, if $N_s(F) 1_{F \in \CE_e^c} \le \E N_T(F)- \eps n$ then $N_T(F) \1_{F \in \CE_e^c} \le \E N_T(F)- \eps n/2$. Thus
\begin{align*}
\P\big(N_s(F) \1_{F \in \CE_e^c} \le \E N_T(F)- \eps n \big) &\le \P\big(N_T(F) \1_{F \in \CE_e^c} \le \E N_T(F)- \eps n/2 \big) \\
&\le \P\big(\CE_e^c \wedge N_T(F) \le \E N_T(F)- \eps n/2 \big) + \P(\CE_e)=o(1),
\end{align*}
where we used Theorem \ref{thm:O:gen}  and Theorem \ref{thm:exceptional}. 

For the second bound regarding $\E(N_T(F) \1_{F \in \CE_e^c})$ , let $N_{us}(F)$ denote the number of roots of $F$ over the set of unstable intervals. By Corollary \ref{cor:manyroots:O}, for non-exceptional $F$ 
 we have that
 $N_{us}(F)\le \eps n/2$, and hence trivially $\E (N_{us}(F) \1_{F \in \CE_e^c}) \le \eps n /2$. Because each $F$ has $O(n)$ roots, we then obtain
\begin{align*}
\E (N_s(F) \1_{F \in \CE_e^c}) &\ge \E N_T(F) - \E (N_{us}(F) 1_{F \in \CE_e^c}) - \E (N_T(F) \1_{F \in \CE_e})\\
&\ge   \E N_T(F) - \eps n/2 - O( n \times  e^{-c\delta^{12} n}) \ge \E N_T(F) - 2\eps n/3.
\end{align*}
\end{proof}

\subsection{Proof of Theorem \ref{thm:orthogonal}: completion} We will justify the result by considering the two tails separately, where we will again rely on the perturbation approach together with the stability results such as  Corollary \ref{cor:stab}.

\subsection{The lower tail}  We need to show that
\begin{equation}\label{eqn:lowertail:O}
\P(  N_T(F)  \le \E N_T(F)  - \eps n/2 ) \le e^{-\Theta(\eps^{18} n)}.
\end{equation}
We will be showing that the complement event $N_T(F)  > \E N_T(F)  - \eps n/2$ has probability at least $1-  e^{-\Theta(\eps^{18} n)}$. With the parameters chosen as in  Corollary \ref{cor:manyroots:O}, consider a non-exceptional polynomial $F$. By Corollary \ref{cor:manyroots:O}, $F$ has at least  $ \E N_T(F) - 2\eps n/3$ roots over the stable intervals.

Let $g$ be an eigenfunction with $\|g\|_2 \le \tau$, where $\tau$ is chosen as in \eqref{eqn:parameters:O}.  Consider a stable interval $I_j$ with respect to $F$ (there are at least $(1/R -\delta)n$ such intervals). We first notice that the number of stable intervals $I_j$ over which $\max_{x \in 3I_j} |g(x)| > \al$ is at most at most $O(\delta n)$. Indeed, assume that there are $M$ such intervals $3I_j$. Then we can choose $M/6$ such intervals that are $R/n$-separated. By Lemma \ref{lemma:largesieve} we have $(M/6) \al^2 \le n \tau^2$, which implies $M \le 6n (\tau \al^{-1})^2 =O(\delta n)$. From now on we will focus on the stable intervals with respect to $F$ on which $|g(\cdot)|$ 
is smaller than $\al$. 


By Corollary \ref{cor:stab} (applied to $I= 3I_j$ with $\mu=\al$ and $\nu =\beta n$, note that $\al/ \beta \asymp \delta^{3/4} <R$), because $\max_{x \in 3I_j} |g(x)| < \al$, the number of roots of $F+g$  over each interval $I_j$ is at least as that of $F$. Hence if $F$ is such that $N_T(F) \ge \E N_T(F) - \eps n/2$ and also $F$ has at least $ \E N_T(F) - 2\eps n/3$ roots over the stable intervals, then by Corollary \ref{cor:manyroots:O}, with appropriate choice of the parameters, $F$ has at least $ \E N_T(F) - \eps n$ roots over the stable intervals $I_j$ above where $|g(\cdot)| \le \al$, 
and hence Corollary \ref{cor:stab} implies that $F+g$ has at least $\E N_T(F)-\eps n$ roots over these stable intervals $I_j$. In particularly $F+g$ has at least $\E N_T(F)-\eps n$ roots over $\BT$. Let $\CU^{lower}$ be the collection of $\Bv_{F}$ from such $F$ (where $N_T(F)\ge \E N_T(F) - \eps n/2$ and $F$ has at least $ \E N_T(F) - 2\eps n/3$ roots over the stable intervals).  Then by Theorem \ref{thm:O:gen} and Corollary \ref{cor:control:O} 
\begin{equation}\label{eqn:Uast:O}
\P(\Bv_{F} \in \CU^{lower}) \ge  1- \P\big(N_T(F)\le \E N_T(F) - \eps n /2\big) - \P\big(N_s(F) \1_{F \in \CE_e^c} \le \E N_T(F)- 2\eps n /3 \big)  \ge 1/2.
\end{equation}

 \begin{proof}(of Equation \eqref{eqn:lowertail:O}) By our application of Corollary \ref{cor:stab} above, the set $\{\Bv, d_2(\Bv,\CU^{lower}) \le \tau\sqrt{2n}\}$ is contained in the set of having at least $\E N_T(F) - \eps n$ roots. Furthermore, \eqref{eqn:Uast:O} says that $\P(\Bv_{F} \in \CU^{lower}) \ge 1/2$. Hence by Theorems \ref{thm:sobolev} and \ref{thm:bounded}
$$
\P( N_T(F)  \ge \E N_T(F)  - \eps n/2) \ge \P\Big(\Bv_{F}\in \big\{\Bv, d_2(\Bv,\CU^{lower}) \le \tau\sqrt{n}\big \}\Big) \ge 1-e^{-\Theta(\eps^{18} n)},
$$
where we used the fact that $\tau \asymp \delta^4$ from \eqref{eqn:parameters:O} and $\delta$ from \eqref{delta:eps:O}.
 \end{proof}

\subsection{The upper tail}
Our goal here is to justify the upper tail
\begin{equation}\label{eqn:uppertail}
\P(  N_T(F)  \ge \E N_T(F)  + \eps n/2 ) \le e^{-\Theta(\eps^{18} n)}.
\end{equation}

Let $\CU^{upper}$ denote the set of $\Bv_{F}$ for which  $N_T(F) \ge E N_T(F)+ \eps n/2$. By Theorem \ref{thm:exceptional} it suffices to assume that $F$ is non-exceptional.

 
 \begin{proof}(of Equation \eqref{eqn:uppertail}) Assume that for a non-exceptional $F$ we have $N_T(F) \ge \E N_T(F) + \eps n$. Then by  Corollary \ref{cor:manyroots:O} the number of roots of $F$ over the stable intervals is at least $ \E N_T(F) +\eps n/3$. Let us call the collection of $\Bv_{F}$ of these polynomials by $\CS^{upper}$. Then argue as in the previous subsection (with the same parameters of $\al, \beta, \tau, \delta$), Corollary \ref{cor:manyroots:O} and Corollary \ref{cor:stab} imply that any $h=F+g$ with $\|g\|_2 \le \tau$ has at least $\E N_T(F) +  \eps n /4$ roots. On the other hand, we know by Theorem \ref{thm:O:gen} that the probability that $F$ belongs to this set of functions is smaller than $1/2$. It thus follows by Theorems \ref{thm:sobolev} and  \ref{thm:bounded} that 
$$\P(\Bv_{F}\in  \CU^{upper}) \le  e^{-\Theta(\eps^{18} n)},$$
where we again used the fact that $\tau \asymp \delta^4$ from \eqref{eqn:parameters:O} and $\delta$ from \eqref{delta:eps:O}.
\end{proof}

\section{Elliptic polynomials: proof of Theorem \ref{thm:Elliptic}}\label{section:elliptic}


Now we turn to Elliptic polynomials. Generally speaking, we will follow the same approach as in the proof of Theorem \ref{thm:orthogonal}.  However to obtain variants of the ``Non-degenerating" property 
(such as \eqref{eqn:non-degenerating} from Lemma \ref{lemma:smallball}) and Markov-Bernstein type estimate (such as \eqref{eqn:MB:gen} from Proposition \ref{prop:manyroots}) 
conditions one has to elaborate much more. Additionally, one of the key difficulties in this model is at the perturbation step, where we won't perturb the coefficients of the polynomials coordinatewise, but only after a projection. For convenience, we will assume that $n$ is an odd number, the treatment for $n$ even is similar and is omitted.

 \subsection{Properties of Elliptic polynomials}\label{subsection:supporting:E}
 First  consider the normalized version 
 \begin{equation}\label{eqn:F:E}
 F(x)=\sum_{i=1}^n a_i \frac{1}{(1+x^2/n)^{n/2}} \sqrt{\binom{n}{i}} (x/\sqrt{n})^i.
 \end{equation}
Let 
$$\frac{x/\sqrt{n}}{\sqrt{x^2/n+1}} =\cos \theta, \frac{1}{\sqrt{x^2/n+1}}=\sin(\theta), \theta \in T'= [0,\pi].$$
To find roots of $F(x)$, we will instead find roots $\theta \in [0,\pi]$ of $G(\theta)=0$ where
\begin{equation}\label{eqn:G:E'}
G(\theta) = F(\sqrt{n} \cot(\theta)) =\sum_{i=1}^n a_i  \sqrt{\binom{n}{i}}\cos^i(\theta) \sin^{n-i}(\theta).
\end{equation}
As $x = \sqrt{n} \cot(\theta)$ by the chain rule,
$$G'(\theta) = \frac{dG}{d\theta} = \left(\frac{dF}{dx}\right) \left(\frac{dx}{d\theta}\right) = \sqrt{n} \left(\frac{d F}{dx}\right) \times \left(\frac{d \cot(\theta)}{d\theta}\right) = \sqrt{n} F'(\sqrt{n} \cot(\theta)) \times \left(\frac{-1}{\sin^2 \theta}\right).$$
In other words,
\begin{equation}\label{eqn:G'}
G'(\theta)\sin^2(\theta)/\sqrt{ n} = -F'(\sqrt{n} \cot(\theta)).
\end{equation}

To put it in a general context, and motivated by the form \eqref{eqn:G:E'}, we will consider the following bivariate homogeneous polynomial of degree $n$ in $\CP_{1,n}$ (where $x_0 = \sqrt{x_0^2 +x_1^2} \cos \theta$ and $x_1= \sqrt{x_0^2 +x_1^2} \sin \theta$)
$$H(x_0,x_1) = \sum_{i=1}^n a_i  \sqrt{\binom{n}{i}} x_0^i x_1^{n-i}.$$
The Bombieri-Weyl norm over $\CP_{1,n}$ 
of $H$ is defined as 
$$\|H\|_{BW} := \sum_i \left(a_i  \sqrt{\binom{n}{i}}\right)^2 \binom{n}{i}^{-1} = \sum_i a_i^2.$$
As $G(\theta)$ is a restriction of $H(x_0,x_1)$ onto $S^1$, we will also say that $\|G\|_{BW} =  \sum_i a_i^2$.

For $l=0,\dots, n$, consider the space $\CH_{1,l} \subset \CP_{1,l}$ of homogeneous harmonic polynomials of degree $l$ (i.e. $\Delta_{\R^2} H=0$ for $H \in \CH_{1,l}$). We know that (see \cite[Eq. 2.1]{DL}) $\CP_{1,n}$ can be decomposed as 
\begin{equation}\label{eqn:PH}
\CP_{1,n} = \bigoplus_{n-l \in 2\Z} \|x\|_2^{n-l} \CH_{1,l}.
\end{equation}
For each $l \in \BN$, the restriction of $\CH_{1,l}$ to the unit circle is the space $V_{1,l}$ of Laplacian eigenfunctions of $S^1$ 
$$V_{1,l}= \Big\{h\in S^1 \to \R, \quad \Delta_{S^1}(h) = -l^2 h \Big\} = \Big\{ a\cos (l \theta) + b \sin (l \theta),\quad  a,b \in \R\Big\}.$$ 
Let 
$$w_{1,n}(l) = \Big[\Vol(S^1) \Gamma(1) \frac{\Gamma((n+l)/2 +1)}{\Gamma(1+(n+l)/2)} \left(\frac{1}{2^n}\right) \binom{n}{(n-l)/2}\Big]^{1/2}=\Big[\left(\frac{1}{2^n}\right) \binom{n}{(n-l)/2}\Big]^{1/2}.$$
It is known (see \cite[Eq. 2.3 and Remark 5]{DL}) that $\{w_{1,n}(l) \cos(l \theta), w_{1,n}(l) \sin(l \theta)\}_{l=1}^n$ 
also form an orthonormal basis of $V_{1,l}$ with respect to the BW inner product. 
Hence if we write 
\begin{equation}\label{eqn:Gbc}
G(\theta) = \sum_{n-l\in 2\Z}  b_l \Big[\left(\frac{1}{2^n} \right)\binom{n}{(n-l)/2}\Big]^{1/2} \cos(l \theta)+c_{l} \Big[\left(\frac{1}{2^n}\right) \binom{n}{(n-l)/2}\Big]^{1/2} \sin(l \theta),
\end{equation}
then we obtain the following
\begin{equation}\label{eqn:Uabc}
\sum_l b_l^2 +c_l^2 = \sum_l a_l^2.
\end{equation}

Putting together, because of the two orthonormal basis (with respect to the BW inner product), we have thus obtained the following deterministic lemma.

\begin{lemma}\label{lemma:Bern:E} There is an orthogonal matrix $U$ of size $n+1$ such that (assuming $n$ odd) for any real numbers $a_0,\dots, a_n$ 
with $G(\theta) = \sum_{i=0}^n a_i  \sqrt{\binom{n}{i}}\cos^i(\theta) \sin^{n-i}(\theta)$ we have 
$$\|G\|_2^2 = \int_{0}^{\pi} G(\theta)^2 d\theta = \sum_{l}   \left(\frac{1}{2^n} \right) \binom{n}{(n-l)/2}(b_l^2 +c_l^2)$$
where 
$$(b_1,c_1,b_3,c_3,\dots, b_n,c_n)^T= U(a_0,\dots, a_n)^T.$$ 
\end{lemma}

For short,  we set 
\begin{equation}\label{eqn:N:E}
N:= \sqrt{n}, \quad  C^\ast := (\log n)^2, \quad  N^\ast := C^\ast N.
\end{equation}
Next we write $G=G_1+G_2$ where 
$$G_1(\theta) := \sum_{\substack{0\leq l \leq N^{\ast}\\ n-l\in 2\Z}}  \big[\left(\frac{1}{2^n} \right) \binom{n}{(n-l)/2}\big]^{1/2}(b_{l} \cos(l \theta)+c_{l} \sin(l \theta)).$$
Then we see that (as for $l > N^\ast$ we easily have $(1/2^n) \binom{n}{(n-l)/2}\le \exp(-\log^{4}n)$)
$$\|G_2\|_2 \le \frac{\sum_l a_l^2}{\exp(\log^{4}n)}.$$
And in general, by Bernstein's inequality for trigonometric polynomials (see also Lemma \ref{lemma:Bern:orth}), for any $d$ we have 
\begin{equation}\label{eqn:E:Bern}
\|G_2^{(d)}\|_2 \le (Cn)^d \|G_2\|_2 \le (Cn)^d \frac{\sum_l a_l^2}{\exp(\log^{4}n)}; \quad \mbox{ and } \quad 
\|G_1^{(d)}\|_2 \le N^\ast \|G_1^{(d-1)}\|_2 \le \dots \le  (N^\ast)^d \|G_1\|_2.
\end{equation}

Next we consider the random polynomial $G(\theta)=  \sum_{i=1}^n \xi_i  \sqrt{\binom{n}{i}}\cos^i(\theta) \sin^{n-i}(\theta)$ of form \eqref{eqn:G:E'}. As before, in both the boundedness and the log-Sobolev cases, the event $\CE_{b}$ that $\frac{\sum_i \xi_i^2}{n}=O(1)$ has probability at least $1 -\exp(-c n)$, so we will assume that
\begin{equation}\label{eqn:E_{b,E,1}}
\sum_l \xi_l^2 = O(n).
\end{equation}

We will need to control $\|G\|_2$. Note that with $\bxi=(\xi_0,\dots, \xi_n)$, where $\xi_i$ are as in Theorem \ref{thm:Elliptic}, and with $(b_1,c_1,b_3,c_3,\dots, b_n,c_n)^T= U \bxi^T$ we have $G=G_1+G_2$, where we have learned that (under $\CE_b$) $\|G_2\|_2 \le \frac{n}{\exp(\log^{4}n)}$ and by definition 
\begin{equation}\label{eqn:upper:L2:E}
\|G_1\|_2^2 = \sum_{\substack{0\leq l \leq N^{\ast}\\ n-l\in 2\Z}}  \big[\left(\frac{1}{2^n} \right) \binom{n}{(n-l)/2}\big](b_{l}^2 +c_{l}^2) \le  \frac{C}{\sqrt{n}}  \sum_{\substack{0\leq l \leq N^{\ast}\\ n-l\in 2\Z}} (b_{l}^2 +c_{l}^2) .
\end{equation}
By definition, we can write $(b_1,c_1,\dots, b_{N^\ast}, c_{N^\ast})^T = U' \bxi^T$ where $U'$ is the submatrix of $U$ of the first ${N^\ast}$ row vectors $\Bv_1,\dots, \Bv_{{N^\ast}}$ (which are orthogonal and each has unit $L_2$-norm).  

\begin{lemma}[Controlling of the $L_2$ norm]\label{lemma:E:L_2} Assume that $\Bv_1,\dots, \Bv_{N^\ast}$ are deterministic orthogonal unit vectors. With $\bxi=(\xi_0,\dots, \xi_n)$ as in Theorem \ref{thm:Elliptic} we have
$$\P(\langle \Bv_1,\bxi \rangle^2+\dots +\langle \Bv_{N^\ast},\bxi \rangle^2 \asymp {N^\ast}) = 1- \exp(-\Theta({N^\ast})).$$
\end{lemma}
\begin{proof}(of Lemma \ref{lemma:E:L_2}) We have $\|U'\bxi\|_2$ is 1-Lipschitz of mean $\E (\|U'\bxi\|_2^2) = N$. Hence by Talagrand concentration inequality (or by Theorem \ref{thm:sobolev})
$$\P(\|U'\bxi\|_2  \asymp \sqrt{{N^\ast}}) =1 - \exp(-\Theta({N^\ast})).$$
\end{proof}

Let $\CE_{b,E}$ denote the intersection of the event \eqref{eqn:E_{b,E,1}} and the event from Lemma \ref{lemma:E:L_2}, on which we have 
\begin{equation}\label{eqn:b,E}
\|G\|_2^2 = O({N^\ast}/\sqrt{n})= O(C^\ast).
\end{equation} 
By \eqref{eqn:E:Bern} and \eqref{eqn:upper:L2:E} we thus have

\begin{proposition}\label{prop:manyroots:E} Under $\CE_{b,E}$, Proposition \ref{prop:manyroots} holds for $G(\theta)$ over $[0,\pi]$ with $C^\ast = \log^2 n$ and $N$ as in \eqref{eqn:N:E}.
\end{proposition}

Now for the non-degenerating 
condition, we first use \cite[Proposition 13.3]{TV} where it was shown that, with 
$$\Bv=\Big(\frac{1}{(1+x^2/n)^{n/2}} \sqrt{\binom{n}{i}} (x/\sqrt{n})^i\Big)_{i=1}^n \mbox{ and } 
\Bv'=\Big(\partial[\frac{1}{(1+x^2/n)^{n/2}} \sqrt{\binom{n}{i}} (x/\sqrt{n})^i]/\partial x\Big)_{i=1}^n$$
then we have 
$$|\Bv \wedge \Bv'| \asymp 1.$$
We can then use \eqref{eqn:G'} and the proof of Theorem \ref{thm:repulsion} to obtain

\begin{theorem}\label{thm:repulsion:E}
 For any fixed $\theta \in (0,\pi)$, for $\al,\beta \ge 1/\sqrt{n}$
$$\P(|G(\theta)| \le \al, |G'(\theta)\sin^2(\theta)/\sqrt{ n}| \le \beta)=O(\al \beta).$$
In particular, for any fixed $\theta \in [\eps_0,\pi-\eps_0]$
$$\P(|G(\theta)| \le \al, |G'(\theta)/\sqrt{ n}| \le \beta)=O(\al \beta/\eps_0^2).$$

\end{theorem}
This is the part where we have to assume $\sin(\theta)$ is away from zero to be effective, in other words we will work within the bulk $T=[\eps_0, \pi- \eps_0]$ from here on (where $\eps_0 \asymp 1/C_E$ with $C_E$ from Theorem \ref{thm:Elliptic}), over which 
$$\P(|G(\theta)| \le \al, |G'(\theta)/\sqrt{ n}| \le \beta)=O_{C_E}(\al \beta),$$
In fact we can still choose $\eps_0$ to have order $n^{-c}$ for some small constant $c$ and choose $\al, \beta$ to be slightly smaller than that (as long as they are at least $n^{-1/2}$ and satisfy a few more constraints to be mentioned below), but we will not dwell into this.

\subsection{Exceptional polynomials}\label{subsection:exceptional:E} This current subsection is similar to Subsection \ref{subsection:exceptional:O}, 
except here that the perturbation will be different. In what follows 
$$T = [\eps_0, \pi - \eps_0],\quad \quad  T'=[0,\pi].$$
and
\begin{equation}\label{eqn:G:E}
G(\theta) = \sum_{i=0}^n \xi_i \sqrt{\binom{n}{i}}\cos^i(\theta) \sin^{n-i}(\theta).
\end{equation}
Note that by Lemma \ref{lemma:E:L_2}, under $\CE_{b,E}$ we have
$$ \int_{0}^{\pi} G(\theta)^2 d\theta \le \frac{{N^\ast}}{\sqrt{n}} = \log^2 n = C^\ast.$$

As before, for convenience, for each $G(\theta)$ we assign a unique (unscaled) vector $\Bv_{G}= (\xi_i)$ in $\Omega^{n+1}$, which is a random vector when $G$ is random. 

Fix $R>4$. Cover $T$ by $N/R$ open intervals $I_i$ of length (approximately) $R/N$ each. Let $3 I_i$ be the interval of length $3/N$ having the same midpoint with $I_i$. Given some parameters $\al, \beta$, we call an interval $I_i$ {\it stable} for a function $f \in \CE_{b,E}$ of form \eqref{eqn:G:E} if there is no point in $x\in 3I_i$ such that $|f(\theta)|\le \al$ and $|f'(\theta)|\le \beta N$.  Let $\delta$ be another small parameter (so that $\delta R <1/4$), we call $f$ {\it exceptional} if the number of unstable intervals  is at least $\delta N$. We call $f$ not exceptional otherwise.

Let $\CE_e = \CE_e(\al,\beta; \delta)$ denote the set of vectors $\Bv_{f}$ associated to exceptional function $f$. Our goal in this section is the following analog of Theorem \ref{thm:exceptional}.

\begin{theorem}\label{thm:exceptional:E} Assume that $\al,\beta,\delta$ satisfy 
\begin{equation}\label{eqn:parametersthm}
\al \asymp \delta^{2}, \beta \asymp \delta^{3/4}, N^{-1/6} < \delta < (\log n)^{-16}.
\end{equation}
Then we have
$$\P\Big(\Bv_{G} \in \CE_e\Big) \le 2 e^{-\delta^5 \sqrt{n}/4C_0}.$$
\end{theorem}

The proof is similar to that of Theorem  \ref{thm:exceptional}, except here $g$, the perturbing  
function (vector) is different. First assume that $f$ (playing the role of $G$) is exceptional, then there are $K=\lfloor \delta n/3 \rfloor$ unstable intervals that are $R/n$-separated.

 Now for each unstable interval in this separated family we choose $x_j \in 3 I_j$ where  $|f(x_j)|\le \al$ and $|f'(x_j)|\le \beta N$ and consider the interval $B(x_j, \gamma/N)$ for some $\gamma <1$ chosen sufficiently small (given $\delta$). Let 
$$M_j:= \max_{x\in B(x_j,\gamma/N)} |f''(x)|.$$  
By Lemma \ref{lemma:largesieve} (with $N=\sqrt{n} $, and the conditions are satisfied by \eqref{eqn:E:Bern} and $\CE_{b,E}$), for $C^\ast =\log^{2} n$ we have
$$\sum_{j=1}^K M_j^2 \le (C^\ast N) \int_{x \in T} f''(x)^2 dx \le  (C^\ast N)^5  \int_{x \in T'} f(x)^2 dx \le (C^\ast)^6 N^5  .$$
Let $C_1$ be sufficiently large. We thus infer from the above that the number of $j$ for which $M_j \ge C_1 \delta^{-1/2}(C^\ast)^{3} N^2$ is at most $C_1^{-2} \delta n$. Hence for at least $(1/3 - C_1^{-2})\delta N$ indices $j$ we must have $M_j <C_1 \delta^{-1/2} (C^\ast)^{3} N^2$.

Consider our function over $B(x_j, \gamma/N)$, then by Taylor expansion of order two around $x_j$, we obtain for any $x$ in this interval
$$ |f(x)| \le |f(x_j)| + |f'(x_j)| (\gamma/N) + \max_{t\in B(x_j, \gamma/N)}|f''(t)| (\gamma/n)^2 \le \al + \beta \gamma + C_1 \delta^{-1/2} (C^\ast)^{3} \gamma^2/2 $$
and similarly
$$|f'(x)| \le (\beta + C_1 \delta^{-1/2} (C^\ast)^{3} \gamma) N.$$

Next recall $\Bv_1,\dots, \Bv_N$ from the previous section. Consider
$$g(\theta) =\sum_{i=0}^n \xi_i' \sqrt{\binom{n}{i}}\cos^i(\theta) \sin^{n-i}(\theta)$$
where we assume that 
\begin{equation}\label{eqn:g:E}
\sum_{i=1}^N (\bxi' \cdot \Bv_i)^2 \le \tau^2 \sqrt{n}.
\end{equation}
We note that this is where our treatment is different from the random orthogonal polynomial cases, and is the main reason why we obtain $\exp(-\Theta(\sqrt{n}))$ in Theorem \ref{thm:exceptional:E}. 

Then it follows from \eqref{eqn:upper:L2:E} (and \eqref{eqn:E:Bern}) that
$$\|g\|_2^2 = \int_{0}^{\pi} g(\theta)^2 d\theta  \le \tau^2.$$
As the intervals $B(\theta_j, \gamma/N)$ are $4/N$-separated, by Lemma \ref{lemma:largesieve} we have
$$\sum_j \max_{x \in B(\theta_j, \gamma/N)} g(\theta)^2 \le  N  \int (g(\theta))^2 d\theta  \le \tau^2N$$
and 
$$\sum_j \max_{\theta \in B(\theta_j, \gamma/N)} g'(\theta)^2   \le N \int (g'(\theta))^2 d\theta \le N^3 \int (g(\theta))^2 d\theta  \le \tau^2 N^3.$$
Hence, again by an averaging argument, with $C_2$ sufficiently large, the number of intervals where either $\max_{\theta \in B(\theta_j, \gamma/N)} |g(\theta)| \ge C_2 \delta^{-1/2} \tau$ or   $\max_{\theta \in B(\theta_j, \gamma/N)} |g'(\theta)| \ge C_2 \delta^{-1/2} \tau N$ is bounded from above by $(1/3 - 2 C_2^{-2})\delta N$. 

On the remaining at least $(1/3 - C_1^{-2} - 2 C_2^{-2})\delta N$ intervals, with $h=f+g$, we have simultaneously that 
$$|h(\theta)| \le  \al + \beta \gamma + C_1 \delta^{-1/2} (C^\ast)^{3} \gamma^2/2  + C_2 \delta^{-1/2} \tau \quad \mbox{ and } \quad |h'(\theta)| \le  (\beta + C_1 \delta^{-1/2}(C^\ast)^{3}  \gamma + C_2 \delta^{-1/2} \tau) N.$$
For short, let 
$$\al':=  \al + \beta \gamma + C_1 \delta^{-1/2} (C^\ast)^{3} \gamma^2/2  + C_2 \delta^{-1/2} \quad \mbox{ and }\quad  \beta':=\beta + C_1 \delta^{-1/2} (C^\ast)^{3} \gamma + C_2 \delta^{-1/2} \tau.$$ 
It follows that $\Bv_h$ belongs to the set $\CU=\CU(\al, \beta,\gamma,\delta, \tau, C_1,C_2)$ 
in $\Omega^{n+1}$ of the vectors  corresponding to $h$, for which the measure of $x$ with $|h(x)| \le  \al'$ and $|h'(x)| \le  \beta' N $ is at least $(1/3 - 2 C_2^{-2})\delta \gamma$ (because this set of $x$ contains $(1/3 - C_1^{-2} - 2 C_2^{-2}) \delta N$ intervals of length $2\gamma/N$). 
Putting together we have obtained the following claim.

\begin{claim}
\label{claim2} Assume that $\Bv_{G} \in \CE_e \cap \CE_{b,E}$ (where we recall $\CE_{b,E}$ from Proposition \ref{prop:manyroots:E}). Then for any $g$ corresponding with $\bxi'=(\xi_0',\dots, \xi_n')$ with $\sum_{i=1}^N (\bxi' \cdot \Bv_i)^2 \le \tau^2 \sqrt{n}$ we have $\Bv_{G+g} \in \CU$. 
\end{claim}
We next argue as in the proof of \eqref{eqn:U:O}, by relying on Theorem \ref{thm:repulsion:E} and by choosing the parameters so that $\al' \beta'$ is much smaller than $\delta \gamma$ (for instance $\gamma \asymp \delta^{5/4}, \tau \asymp \delta^{5/2}$ and $N^{-1/6} < \delta < (\log n)^{-16}$) we have 
\begin{equation}\label{eqn:U:E}
\P(\Bv_{G} \in \CU) \le 1/2.
\end{equation}

\begin{proof}(of Theorem \ref{thm:exceptional:E}) By Theorem \ref{thm:E:concentration} and \eqref{eqn:U:E}, 
$$\P(\Bv\in \CE_e) \le 2 e^{-(\tau n^{1/4})^2/4C_0}.$$
\end{proof}

\subsection{Roots over unstable intervals}\label{subsection:lowertail:E}

Let $0<\eps<1$ be given.  Let $\delta$ be chosen so that 
\begin{equation}\label{delta:eps:E}
\delta =O(\min\{\eps/\log(1/\eps), \eps/(C^\ast)^6\}),
\end{equation}
with $C^\ast = \log^2 n$ as in the previous subsection. Also let $\la=\delta^4$. Let $\al , \beta$ be chosen as in Theorem \ref{thm:exceptional:E}. We clearly have
$$\la (C^\ast)^{1/2} \le \al,  \la (C^\ast)^{3/2} \le \beta.$$
With these parameters, the conditions in Proposition \ref{prop:manyroots} are satisfied (because of \eqref{eqn:E:Bern}). 
Hence we obtain the following consequence (similarly to Corollary \ref{cor:manyroots:O}).

\begin{cor}\label{cor:manyroots:E} With the parameters as above, a non-exceptional $f \in \CE_{b,E}$ cannot have more than $\eps N/2$ roots over any $\delta N$ intervals $I_i$. In particular, $f$ cannot have more than $\eps N/2$ roots over the unstable intervals. 
\end{cor}

For each $G$ of the form \eqref{eqn:G:E} that is not exceptional we let $S(G)$ be the collection of intervals over which $G$ is stable. Let $N_s(G)$ denote the number of roots of $G$ over the set $S(G)$ of stable intervals.

\begin{cor}\label{cor:control:E} With the same parameters as in  Proposition \ref{prop:manyroots}, with $C^\ast = \log^2 n$ and $N=\sqrt{n}$ we have 

$$\P\Big(N_s(G) \1_{G \in \CE_e^c \cap \CE_{b,E}} \le \E N_T(G)- \eps N \Big)=o(1)$$
and
$$\E \Big(N_s(G) \1_{F \in \CE_e^c  \cap \CE_{b,E} }\Big) \ge \E N_T(G)- 2 \eps N/3.$$
\end{cor}
\begin{proof}(of Corollary \ref{cor:control:E}) The proof is similar to that of Corollary \ref{cor:control:O}. For the first bound, by Proposition \ref{prop:manyroots:E}, if $N_s(G) \1_{G \in \CE_e^c  \cap \CE_{b,E} } \le \E N_T(G)- \eps N$ then $N_T(G) \1_{G \in \CE_e^c  \cap \CE_{b,E} } \le \E N_T(G)- \eps N/2$. Thus 

\begin{align*}
\P\big(N_s(G) \1_{G \in \CE_e^c  \cap \CE_{b,E} } \le \E N_T(G)- \eps N \big) &\le \P\big(N_T(G) \1_{G \in \CE_e^c  \cap \CE_{b,E} } \le \E N_T(G)- \eps N/2 \big) \\
&\le \P\big(\CE_e^c  \cap \CE_{b,E}  \wedge N_T(G) \le \E N_T(G)- \eps N/2 \big) + \P(\CE_e)=o(1),
\end{align*}

where we used Theorem \ref{thm:E:gen} (applied to the interval $[-C_E \sqrt{n}, C_E \sqrt{n}]$ in place of $\R$) and Theorem \ref{thm:exceptional:E}. Our treatment for the second bound regarding $\E(N_T(G) 1_{G \in \CE_e^c  \cap \CE_{b,E} })$ is identical to that case of Corollary \ref{cor:control:O}, and hence we omit the details.


\end{proof}

Now we complete the proof of Theorem \ref{thm:Elliptic}. As the proof is similar to that of Theorem \ref{thm:orthogonal}, we will just sketch the ideas. We will split into the lower tail 
$$\P(  N_T(G)  \le \E N_T(G)  - \eps N/2 ) \le e^{-\eps^{O(1)}\sqrt{n}/ \log^{O(1)}n}$$ 
and upper tail 
$$\P(  N_T(G)  \ge \E N_T(G)  + \eps N/2 ) \le e^{-\eps^{O(1)}\sqrt{n}/ \log^{O(1)}n}.$$
In both events, it suffices to focus on non-exceptional $G$ and  $G\in \CE_{b,E}$ (with an additional term $\P(\CE_{b,E}^c) = \exp(-\Theta(N))$ in probability), for which by Corollary \ref{cor:manyroots:E} the number of roots over unstable intervals is negligible. 
 
\begin{itemize}

\item For the lower tail, in the complement event that $N_T(G) \ge \E N_T(G)  - \eps N/2$, the number of roots over stable intervals is at least $N_T(G) \ge \E N_T(G)  - 2\eps N/3$. Then Lemma \ref{lemma:stab} applied to the stable intervals where $|g(x)|<\al$ yields that $N(G+g)$ is at least $\E N_T(G)  - \eps N$ for any $g$ satisfying \eqref{eqn:g:E}. By Theorem \ref{thm:E:gen} the event $N_T(G) \ge \E N_T-\eps N$ is at least 1/2, hence by Theorem \ref{thm:E:concentration} we infer that the original event $N_T(G) \ge \E N_T(G)  - \eps N/2$ has probability at least $1-e^{-\eps^{O(1)}\sqrt{n}/ \log^{O(1)}n}$.\\
\vskip .1in

\item For the upper tail, if $G$ is non-exceptional which has more than $\E N_T(G)  + \eps N/3$ over the stable intervals, then by Lemma \ref{lemma:stab} we also have $N(G+g)$ is at least $\E N_T(G)  + \eps N/3$ for any $g$ satisfying \eqref{eqn:g:E}. On the other hand, by Theorem \ref{thm:exceptional:E} this event has probability smaller than $1/2$, and hence we can apply Theorem \ref{thm:E:concentration} to obtain a bound of type $e^{-\eps^{O(1)}\sqrt{n}/ \log^{O(1)}n}$ for the original event.
\end{itemize}

\section{Weyl polynomials: proof of Theorem \ref{thm:Weyl}}\label{section:Weyl}


In this section we focus on Weyl polynomials. Here, unlike in the orthogonal polynomial model or the Elliptic polynomial model that we related to trigonometric polynomials one way or another,  we will have to modify the function by scaling and truncating to see the main behavior. In the sequel we will be working with
\begin{equation}\label{eqn:F:W}
F(x) = \sum_{i=0}^n \xi_j e^{-x^2/2} \frac{x^i}{\sqrt{i!}}.
\end{equation}

 \subsection{Properties of Weyl polynomials}\label{subsection:supporting:W} We first start with the following estimate.

\begin{claim}\label{claim:Weyl:ix} Assume that $x>0$ is sufficiently large. There exist absolute constants $c_1, c_2$ such that the followings hold.

\begin{itemize}
\item Let $i$ is a positive integer with $i=x^2+t$, where $L= |t|/x \leq x^{\frac{1}{3}}$. Then
\begin{align*} c_1^{-1} x^{-1/2} \exp(-c_1 L^2)  \le e^{-x^2/2} x^{i}/\sqrt{i!} \leq c_2^{-1}x^{-1/2} \exp(-c_2 L^2).\end{align*}
\item Furthermore, if $i \le x^2 - Lx $ or $i\ge x^2 + Lx$, then 
\begin{align*}e^{-x^2/2} x^i/\sqrt{i!} \leq  c_2^{-1}x^{-1/2} \exp(- c_2 \min\{L^2, x^{2/3}\}).\end{align*}
\end{itemize}
\end{claim}

\begin{proof} We focus on the first claim. By Stirling's approximation, $i! \approx \sqrt{2\pi i}(i/e)^i$, and so 
\begin{align*}\frac{x^{2i}}{i!} \approx \frac{1}{\sqrt{2\pi i}} \left(\frac{ex^2}{i}\right)^i.\end{align*}
Substituting $i=x^2+t$ then 
\begin{align}
(ex^2/i)^i = [ex^2/(x^2+t)]^{x^2+t} &= e^{x^2} e^t (1-\frac{t}{x^2+t})^{x^2+t}
\nonumber\\
& = e^{x^2} e^t  \exp\Big(-[\frac{t}{x^2+t} +(\frac{t}{x^2+t})^2/2 + (\frac{t}{x^2+t})^3/3+\dots] \times ({x^2+t})\Big) \nonumber\\
\label{taylor-exp}&= e^{x^2} \exp\Big(-\frac{1}{2}\frac{t^2}{x^2+t} - \frac{1}{3}\frac{t^3}{(x^2+t)^2}-\dots\Big).
\end{align}
With our choice of $t=Lx$ the $k$-th term in the exponent is of the form
\begin{align*}
 \frac{(Lx)^{k+1}}{(x^2+Lx)^k} = \frac{L^{k+1}x}{(x+L)^k}.
\end{align*}
It can readily be observed that
\begin{align*}\sum_{k=2}^\infty \frac{L^{k+1}x}{(x+L)^k(k+1)} = O(1) \text{ as $L \le x^{1/3}$ and $x \to \infty$}. \end{align*}
Hence
$$(ex^2/i)^i  \asymp e^{x^2}  \exp(-\frac{1}{2}\frac{t^2}{x^2+t} )  \asymp e^{x^2}  \exp(-\Theta(L^2)).$$

For the second claim, it suffices to assume $i \le x^2 - Lx$ and $L \ge x^{1/3}$. Notice that $x^i/\sqrt{i!}$ is increasing as $i$ increases to $x^2$ (because the ratio is $(x^{i+1}/\sqrt{(i+1)!})/(x^i/\sqrt{i!}) = x/\sqrt{i+1}\ge 1$ if $i \le x^2-1$), so for  $i \le x^2 - Lx  \le i_0 = \lfloor x^2 - x^{1/3}x \rfloor,$  
$$e^{-x^2/2} x^i/\sqrt{i!}  \le e^{-x^2/2} x^{i_0}/\sqrt{i_0!} \le c_2^{-1}x^{-1/2} \exp(- c_2 x^{2/3}).$$
 
\end{proof}

In our next estimate we treat with higher order derivaties of $e^{-x^2/2} \frac{x^i}{\sqrt{ i!}}$.

\begin{lemma}\label{lemma:Weyl:derivative} Assume that $x>0$ sufficiently large. Let $d$ be a positive integer such that $d \le \log^{O(1)} x$. Assume that $i-x^2 = Lx$ where $|L|=o(x)$. Then we have
$$\Big| \big(e^{-x^2/2} \frac{x^i}{\sqrt{ i!}}\big)^{(d)}\Big| = O\Big( d! (1+|L|)^d e^{-x^2/2} \frac{x^i}{\sqrt{ i!}}\Big).$$
\end{lemma}

\begin{proof} Recall that 
$$(e^{-x^2/2})^{(k)} = (-1)^k He_k(x) e^{-x^2/2},$$
where $He_k(x), k\ge 0$ are Hermite polynomials.

Hence
\begin{align}\label{W:d:start}
\big(e^{-x^2/2} \frac{x^i}{\sqrt{ i!}}\big)^{(d)} &= \sum_{0\le k \le d} (-1)^k \binom{d}{k} He_k(x) e^{-x^2/2} i\dots (i-(d-k)+1) \frac{x^{i-(d-k)}}{\sqrt{ i!}} \nonumber \\
&=e^{-x^2/2}  \frac{x^i}{\sqrt{ i!}} \sum_{0\le k \le d} (-1)^k \binom{d}{k} He_k(x)  \frac{i\dots (i-(d-k)+1)}{x^{d-k}}.
\end{align}
It suffices to work with the second factor.
\begin{claim} $\sum_{0\le k \le d} (-1)^k \binom{d}{k} He_k(x)  \frac{i\dots (i-(d-k)+1)}{x^{d-k}}$ is the coefficient of $w^d$ in $f(w)=d!(1+\frac{w}{x})^i e^{-xw - \frac{w^2}{2}}$. 
\end{claim}

\begin{proof} Let 
$$t=i/x, z=1/x, l = d-k.$$
We use 
Stirling number of the first kind to write 
$$i\dots (i-l+1) = \sum_{m=0}^{l} s(l,m) i^m.$$ 
More generally,
$$ \frac{i\dots (i-(d-k)+1)}{x^{d-k}} = t(t-z) \dots (t-(l-1)z) =  \sum_{m=0}^{l} s(l,m) t^m z^{l-m}.$$
Hence
$$ \sum_{0\le k \le d} (-1)^k \binom{d}{k} He_k(x)  \frac{i\dots (i-(d-k)+1)}{x^{d-k}} =  \sum_{0\le k \le d} (-1)^k \binom{d}{k} He_k(x)  \sum_{m=0}^{l} s(l,m) t^m z^{l-m}.$$

Note that $He_k(x) = e^{- \frac{D^2}{2}} x^k$ where $D$ is the differential operator with respect to $x$. Hence to simplify $ \sum_{0\le k \le d} (-1)^k \binom{d}{k} He_k(x)  \sum_{m=0}^{l} s(l,m) t^m z^{l-m}$ we just need to apply $e^{- \frac{D^2}{2}}$ (with respect to $x$) to the following expression
\begin{align*}
\sum_{0\le k \le d} (-1)^k \binom{d}{k} x^k  t(t-z) \dots (t-(l-1)z) &=  \sum_{m=0}^{d} \sum_{l=m}^d s(l,m) t^m z^{l-m} (-1)^{d-l} \binom{d}{l}  x^{d-l}\\
&= d!\sum_{m=0}^{d} (t/z)^m \sum_{l=m}^d s(l,m) z^{l} (-1)^{d-l} \frac{1}{l! (d-l)!}  x^{d-l}.
\end{align*}

To this end, observe that $\sum_{l=m}^d s(l,m) \frac{z^{l}}{l!} (-1)^{d-l} \frac{x^{d-l}}{(d-l)!}$ is the coefficient of $w^d$ in the following product (as functions of $w$)
$$[\sum_{l=m}^\infty s(l,m) \frac{z^{l} w^l}{l!}]  [\sum_{k=0}^\infty (-1)^{k} \frac{x^kw^k}{k!}] =  \frac{1}{m!}(\log(1+ zw))^m e^{-xw},$$
where we used the fact that the (exponential) generating function of $ s(l,m) , m\le l$, is $\frac{1}{m!}(\log(1+t))^m$.

It thus follows that $\sum_{l=m}^d s(l,m) z^{l} (-1)^{d-l} \frac{1}{l! (d-l)!}  He_{d-l}(x)$ is the coefficient of $w^d$ in 

$$ e^{-D^2/2} [ \frac{1}{m!}(\log(1+ zw))^m e^{-xw}] =  \frac{1}{m!}(\log(1+ zw))^m e^{-D^2/2} e^{-xw} =  \frac{1}{m!}(\log(1+ zw))^m e^{-xw - \frac{w^2}{2}}$$ 

where we noted that 
$$e^{-D^2/2} e^{-xw} = \sum_{j=0}^\infty \frac{(-1/2)^j}{j!} (e^{-xw})^{(2j)}=e^{-xw - \frac{w^2}{2}}.$$
Hence $\sum_{m=0}^{d} (t/z)^m \sum_{l=m}^d s(l,m) z^{l} (-1)^{d-l} \frac{1}{l! (d-l)!}  He_{d-l}(x)$ is the coefficient of $w^d$ in the sum $\sum_{m=0}^{d} (t/z)^m  \frac{1}{m!}(\log(1+ zw))^m e^{-xw - \frac{w^2}{2}}$, which is also the coefficient of $w^d$ in the infinite sum  
\begin{align*}\sum_{m=0}^{\infty} (t/z)^m  \frac{1}{m!}(\log(1+ zw))^m e^{-xw - \frac{w^2}{2}}&=e^{(t/z) \log(1+ zw)}e^{-xw - \frac{w^2}{2}} \\
&= (1+zw)^{t/z}e^{-xw - \frac{w^2}{2}}=(1+\frac{w}{x})^i e^{-xw - \frac{w^2}{2}}.\end{align*}
\end{proof}
As $f(w)$ is analytic in $w$, by Cauchy formula, the coefficient of $w^d$ in $f(w)$ is $\frac{f^{(d)}(0)}{d!}$,
$$\frac{1}{2\pi i}\int_C \frac{f(w)}{w^{d+1}}dw = \frac{f^{(d)}(0)}{d!},$$ 
where $C$ is the circle around zero of radius $R$ to be specified.

We write $i=x^2+ L x$, noting that $x$ is very large compared to $R=|w|$ and $|L|$. So 
\begin{align*}
    (1+\frac{w}{x})^i e^{-xw - \frac{w^2}{2}} = e^{[\frac{w}{x} + O(\frac{R^2}{x^2})] i}e^{-xw - \frac{w^2}{2}} &= e^{[\frac{w}{x} + O(\frac{R^2}{x^2})] (x^2+Lx)}e^{-xw - \frac{w^2}{2}} \\ &= e^{wL + O(R^2)}e^{-\frac{w^2}{2}}.\end{align*}
Hence 
$$\left|\frac{1}{2\pi i}\int_C \frac{f(w)}{w^{d+1}}dw\right| \le d! R \frac{e^{R|L| + O(R^2)}}{R^{d+1}}=d!\frac{e^{R|L| + O(R^2)}}{R^{d}}.$$
Choosing $R=1/(|L|+1)$, we thus obtain a bound as desired for $|\sum_{0\le k \le d} (-1)^k \binom{d}{k} He_k(x)  \frac{i\dots (i-(d-k)+1)}{x^{d-k}}|$.
\end{proof}

We now proceed with the $L_2$-norm and Markov-Bernstein type estimates. Consider 
\begin{equation}\label{eqn:f:W}
f(x)=\sum_{i=0}^n a_i e^{-x^2/2} \frac{x^i}{\sqrt{i!}},
\end{equation}
where $a_i$ are now deterministic. 

Let $C_W$ be fixed. From now on we let 
\begin{equation}\label{eqn:I:W}
I:= [ \frac{1}{C_W}\sqrt{n}, C_W\sqrt{n}].
\end{equation}
The following (deterministic) lemma allows us to control the ``$L_2$-norm" of the function.

\begin{lemma}\label{lemma:W:d} Assume that $0\le d\le d_0= \log^{O(1)} n$. There exists a constant $C$ such that the following holds
$$\int_I (f^{(d)}(x))^2 dx \ll   (\log^4 n)  [C(\log^2 n) d]^{2d}  \sum_i a_i^2.$$
\end{lemma}


\begin{proof}(of Lemma \ref{lemma:W:d}) We divide the interval $I$ into $O(\sqrt{n})$ subintervals $I_j=[\sqrt{k_j}, \sqrt{k_{j+1}})$, where $k_j = n/C_W + j \sqrt{n}$. Set 
$$L := 8 (\log n)^2.$$
For each $x\in I_j$ we  can break down the sum in Lemma \ref{lemma:W:d} as follows: 
\begin{align}(\sum_{i=0}^{n} e^{-x^2/2} a_i x^i/\sqrt{i!})^{(d)}=(\sum_{i=\lfloor k_j-L\sqrt{n}\rfloor}^{\lfloor k_j+L\sqrt{n}\rfloor} (e^{-x^2/2} a_i x^i/\sqrt{i!}))^{(d)}+(\sum_{i\leq\lfloor k_j-L\sqrt{n}\rfloor}(*))^{(d)}+(\sum_{i\geq\lfloor k_j+L\sqrt{n}\rfloor+1}(*))^{(d)}.
\end{align}

We first show that the contribution in $\int_I (f^{(d)}(x))^2 dx$ of the truncated terms in negligible 
\begin{align*}
& \sum_{j=0}^{\lfloor (C_W-1/C_W) \sqrt{n} \rfloor}  \int_{\sqrt{k_j}}^{\sqrt{k_{j+1}}}  [\sum_{i\leq\lfloor k_j-L\sqrt{n}\rfloor}(e^{-x^2/2} a_i x^i/\sqrt{i!}))^{(d)}]^2 +[\sum_{i\geq\lfloor k_j+L\sqrt{n}\rfloor}(e^{-x^2/2} a_i x^i/\sqrt{i!}))^{(d)}]^2 dx \\
& \le \exp(-\Theta(\log^4 n)) \sum_i a_i^2.
 \end{align*}
Indeed, for instance when $i\ge  k_j + L\sqrt{n}$, we can use the fact that 
$$(e^{-x^2/2} x^i/\sqrt{i!})^{(d)} = e^{-x^2/2} (x^{i-d}/\sqrt{(i-d)!}) P_d(x),$$
where $|P_d(x)| \le (Ci)^{2d} |x|^{2d} \le (C n)^{4d}$ for some sufficiently large constant $C$. On the other hand, by Claim \ref{claim:Weyl:ix}, as $i\ge  k_j + L\sqrt{n}$ and $d=O(\log n)$, we have $ e^{-x^2/2} x^{i-d}/\sqrt{(i-d)!} \le \exp(-\Theta(\log^4 n))$. We can argue similarly for the $i\le k_j -L\sqrt{n}$ case.

It remains to work with the main term
\begin{align}\label{eqn:L_2:mainterm:W}
 \sum_{j=0}^{\lfloor (C_W-1/C_W) \sqrt{n} \rfloor}  \int_{\sqrt{k_j}}^{\sqrt{k_{j+1}}}  (\sum_{i=\lfloor k_j -L\sqrt{n}\rfloor }^{\lfloor k_j+L\sqrt{n}\rfloor} a_i (e^{-x^2/2} \frac{x^i}{\sqrt{i!}})^{(d)})^2 dx.\
 \end{align}

By Lemma \ref{lemma:Weyl:derivative} we see that for $x \in I_j = [\sqrt{k_j}, \sqrt{k_{j+1}})$ and $i \in [\lfloor k_j-L\sqrt{n}\rfloor, \lfloor k_j+L\sqrt{n}\rfloor]$
$$ |(e^{-x^2/2} \frac{x^i}{\sqrt{i!}})^{(d)}| \ll |(Ld)^d  e^{-x^2/2} \frac{x^i}{\sqrt{i!}}|  \ll (Ld)^d \frac{1}{\sqrt{x}}$$
where in the last estimate we used Claim \ref{claim:Weyl:ix}.

Thus by Cauchy-Schwarz 
\begin{align}\label{trunc:L2}
& \sum_{j=0}^{\lfloor (C_W-1/C_W) \sqrt{n} \rfloor} \int_{\sqrt{k_j}}^{\sqrt{k_{j+1}}}  (\sum_{i=\lfloor k_j -L\sqrt{n} \rfloor }^{\lfloor k_j+L\sqrt{n}\rfloor} a_i (e^{-x^2/2} \frac{x^i}{\sqrt{i!}})^{(d)})^2 dx \nonumber \\
& \ll (Ld)^{2d} \sum_{j=0}^{\lfloor (C_W-1/C_W) \sqrt{n} \rfloor} \int_{\sqrt{k_j}}^{\sqrt{k_{j+1}}} (\sum_{i=\lfloor k_j -L\sqrt{n}\rfloor }^{\lfloor k_j+L\sqrt{n}\rfloor} a_i^2)  ( \sum_{i=\lfloor k_j -L\sqrt{n}\rfloor }^{\lfloor k_j+L\sqrt{n}\rfloor}  \frac{1}{x})dx \nonumber  \\
&\ll (Ld)^{2d}  \sum_{j=0}^{\lfloor (C_W-1/C_W) \sqrt{n} \rfloor}  (\sum_{i=\lfloor k_j -L\sqrt{n}\rfloor }^{\lfloor k_j+L\sqrt{n}\rfloor} a_i^2)    \int_{\sqrt{k_j}}^{\sqrt{k_{j+1}}} \frac{2L\sqrt{n}}{x}dx \nonumber \\
&  \ll L (Ld)^{2d} \sum_{j=0}^{\lfloor (C_W-1/C_W) \sqrt{n} \rfloor}  \sum_{i=\lfloor k_j -L\sqrt{n}\rfloor }^{\lfloor k_j+L\sqrt{n}\rfloor} a_i^2  (\sqrt{k_{j+1}}-\sqrt{k_{j}})  \nonumber  \\
& \ll L (Ld)^{2d} \sum_{j=0}^{\lfloor (C_W-1/C_W) \sqrt{n} \rfloor} \sum_{i=\lfloor k_j -L\sqrt{n}\rfloor }^{\lfloor k_j+L\sqrt{n}\rfloor} a_i^2 \ll L^2 (Ld)^{2d}  \sum_{i=0}^n a_i^2,
\end{align}
where we noted 
that for all $x\in I_j$, $x \approx \sqrt{k_j}$, which has order $\sqrt{n}$.

\end{proof}

While Lemma \ref{lemma:W:d} gives us a useful comparison between $\int_I (f^{(d)}(x))^2 dx$  and the $L_2$-norm for any sequence $(a_0,\dots, a_n)$ that will be useful in geometric concentration, in our next lemma we show that for the random case, the bound can be slightly improved.

\begin{lemma}\label{lemma:W:prob:d} 
Let $N= \sqrt{n}$ and $d_0 =C \log n$ for some given constant $C$. Assume that $\Ba=(a_0,\dots, a_n)$ where the $a_i$ are iid copies of a subgaussian $\xi$ of mean zero and variance one. Then for $f$ of the form \eqref{eqn:f:W}, with probability at least $1- \exp(-\Theta(N/\log^2 n))$ with respect to $\Ba$ we have 
\begin{equation}\label{eqn:W_0}
\int_I (f^{d}(x))^2 dx \ll  (\log^4 n)[ C (\log^2 n) d]^{2d} N,\quad \quad \quad  0\le d\le d_0.
\end{equation}
\end{lemma}

We let $\CE_{W,d_0}$ denote the event of $\Ba=(a_0,\dots,a_n)$ satisfying the above bound. Then 
$$\P(\CE_{W,d_0})\ge 1- \exp(-\Theta(N/\log^2 n))$$ and for $\Ba\in \CE_{W,d_0}$ the bound \eqref{eqn:W_0} holds for $f$ associated with $\Ba$. 

Observe that for small $d$ (such as when $d=O(1)$), the above bound is off by a factor $1/N$ compared to the bound from Lemma \ref{lemma:W:d}. Motivated by this crucial fact, we will also be working with the event $\CE_{W,3}$ later.

\begin{proof}(of Lemma \ref{lemma:W:prob:d}) It suffices to work with a given $d$ in the range $0\le d\le d_0$. For short let 
$$g(x):=f^{(d)}(x).$$
It suffices to assume that $\sum_i a_i^2 \ll n$ as this event has probability at least $1 - \exp(-\Theta(n))$. Under this event, by Lemma \ref{lemma:Weyl:derivative} and by the method of truncation above (basing on Claim \ref{claim:Weyl:ix}), for $d=O(\log n)$ it is clear that for all $x\in I$
$$|g'(x)|, |g(x)|\le [C (\log^2 n) d]^{2(d+1)} \sum_i a_i^2 \le    [ C (\log^2 n) d]^{2(d+1)} n.$$
As such, with say $M = [ C (\log^2 n) d]^{4(d+1)} n^2$, we can approximate the integral by finite sum
$$\int_{I_j} g^2(x) dx \approx \frac{|I_j|}{M}\sum_{i=0}^M g^2(x_{jk}),$$
where the points $x_{jk} ,1\le k\le M$ divide $I_j=[\sqrt{k_j}, \sqrt{k_{j+1}})$ into $M$ intervals of the same length. 
So
$$ \int_{I} g^2(x) dx \approx \sum_{j=0}^{ (C_W-1/C_W) \sqrt{n}} \frac{|I_j|}{M}\sum_{i=0}^M g^2(x_{jk}).$$ 
Now for each $k$, by applying Lemma \ref{lemma:Weyl:derivative} and Claim \ref{claim:Weyl:ix}, the main part of $g(x_{jk})$, $\sum_{i=\lfloor k_j -L\sqrt{n}\rfloor }^{\lfloor k_j+L\sqrt{n}\rfloor} a_i (e^{-x^2/2} \frac{x^i}{\sqrt{i!}})^{(d)}$ from \eqref{eqn:L_2:mainterm:W}, is a subgaussian random variable of mean zero and variance at most $ (\log^4 n)  [ C (\log^2 n) d]^{2d}$  in the $\sigma$-algebra generated by $a_k, k \in  \lfloor k_j-L\sqrt{n}\rfloor \le k \le \lfloor k_j+L\sqrt{n}\rfloor$ (this is because each $a_i$ is subgaussian of mean zero and variance one, and by Lemma \ref{lemma:Weyl:derivative}, $| (e^{-x^2/2} \frac{x^i}{\sqrt{i!}})^{(d)}| \le  [ C (\log^2 n) d]^{d} e^{-x^2/2} \frac{x^i}{\sqrt{i!}} \le   [ C (\log^2 n) d]^{d}  (1/\sqrt{x})$ and $x$ has order $\sqrt{n}$). 

Hence, for each fixed $j_0\in \{0,\dots, L-1\}$, the $\lfloor (C_W-1/C_W)\sqrt{n}/L \rfloor$ sub-exponential random variables  $g^2(x_{(n/C_W+j_0)k}), g^2(x_{(n/C_W+j_0+L)k}),g^2(x_{(n/C_W+j_0+2L)k}), \dots$ are independent because they correspond to disjoint $J_j$ such as, and so 
$$\P \Big( g^2(x_{n/C_W+j_0k})+g^2(x_{(n/C_W+j_0+L)k}) + g^2(x_{(n/C_W+j_0+2L)k}) + \dots  \ge  (\log^2 n)  [ C (\log^2 n) d]^{d} \times N  \Big) \le \exp(-N/L).$$

By union bound, the intersection $\CE_{W,d_0}$ of the complement of these $L M$ events has probability at least 
$$\P(\CE_{W,d_0}) \ge 1- L [ C (\log^2 n) d]^{4(d+1)} n^2 \exp(-N/L) \ge 1 - \exp(-N/2L).$$
On the event $\CE_{W,d_0}$, 
\begin{align*}
\int_{I} g^2(x) dx & \approx \sum_{j=0}^{ (C_W-1/C_W) \sqrt{n}} \frac{|I_j|}{M}\sum_{i=0}^M g^2(x_{jk})  \asymp\frac{1}{M}\sum_{i=0}^M \sum_{j=0}^{ (C_W-1/C_W) \sqrt{n}} g(x_{jk})^2 \\
& \le \frac{1}{M}\sum_{i=0}^M L  (\log^2 n)  [ C (\log^2 n) d]^{d} \times N = L(\log^2 n)  [ C (\log^2 n) d]^{d} .
\end{align*}
\end{proof}

We next check the non-degenerating and repulsion property as in Theorem \ref{thm:repulsion}.  
With $\Bv(x) = e^{-x^2/2} (x^i/\sqrt{i!})_{i=0}^n$ we have
$$ \Bv'(x) =  e^{-x^2/2} \left( \frac{i-x^2}{x} \frac{x^i}{\sqrt{i!}}\right)_{i=0}^n.$$ 
The following shows that $\Bv(x)$ and $\Bv'(x)$ are not degenerate over $I = [\frac{1}{C_W} \sqrt{n}, C_W\sqrt{n}]$.

\begin{claim} For $x\in I$ we have 
$$|\Bv(x) \wedge \Bv'(x)| \asymp 1.$$
\end{claim}

\begin{proof} It is because (see also \cite[Section 12]{TV}) 
$$|\Bv(x) \wedge \Bv'(x)|^2 = e^{-2x^2} \sum_{x^2 +c_1x \le i, j\le x^2 +c_2x} \frac{|i-j|^2}{x^2} \frac{x^{2i}}{i!}  \frac{x^{2j}}{j!} \asymp 1.$$
\end{proof}

As a consequence, by Theorem \ref{thm:repulsion} 

\begin{theorem}\label{thm:repulsion:W} For $x \in I$, for any $\al,\beta \gg 1/\sqrt{n}$ we have
$$\P(|\langle \bxi, \Bv(x)\rangle | \le \al, |\langle \bxi, \Bv'(x)\rangle | \le \beta ) \le O(\al \beta).$$
In other words, 
$$\P\left(\sum_{i=0}^{n} \xi_i \big( e^{-x^2/2} \frac{x^i}{\sqrt{i!}},   e^{-x^2/2} \frac{i-x^2}{x} \frac{x^i}{\sqrt{i!}}\big) \in (-\alpha,\al) \times (-\beta, \beta)\right) =O(\al \beta).$$
\end{theorem}

\subsection{Exceptional polynomials}\label{subsection:exceptional:W}

This current subsection is similar to Subsection \ref{subsection:exceptional:E}, except here that the introduction of $g$ will be totally different. Set 
$$N :=C_W\sqrt{n}.$$  
Consider $F(x) =\sum_{i=0}^n a_i e^{-x^2/2} \frac{x^i}{\sqrt{i!}}$, where $x\in I=[\frac{1}{C_W}\sqrt{n}, C_W \sqrt{n}]$.
Consider a rescaling
\begin{equation}\label{eqn:G:W}
G(x) = F(Nx) = \sum_{i=0}^n a_i e^{-(Nx)^2/2} \frac{(Nx)^i}{\sqrt{i!}},
\end{equation}
where now 
$$x\in T=[1/C_W^2,1].$$ 
We will identify $F$ and $G$ with the vector $\Ba \in \Omega^{n+1}$.

By Lemma \ref{lemma:W:prob:d}, with probability at least $1 - \exp(-\Theta(N/\log^2 n))$ it suffices to assume $\Ba\in \CE_{W,d_0}$ where for all $0\le d\le d_0= C\log n$
\begin{equation}\label{MB:W}
\int_T (G_{\Ba}^{(d)}(x))^2 dx  =  N^{2d} \frac{1}{N}\int_I (F_{\Ba}^{(d)}(x))^2 dx  \le    (\log^4 n) (C^\ast N)^{2d},
\end{equation}
where 
$$C^\ast = C' \log^3 n, \mbox{ for some constant $C'$}.$$

Let $R>4$ be fixed. Cover $T$ by $N/R$ open intervals $I_i$ of length (approximately) $R/N$ each. Let $3 I_i$ be the interval of length $3R/N$ having the same midpoint with $I_i$. Given some parameters $\al, \beta$, we call an interval $I_i$ {\it stable} for a function $f \in \CE_{W,d_0}$ of form \eqref{eqn:G:W} if there is no point in $x\in 3I_i$ such that $|f(x)|\le \al$ and $|f'(x)|\le \beta N$.  Let $\delta$ be another small parameter (so that $\delta R <1/4$), we call $f \in \CE_{W,d_0}$ {\it exceptional} if the number of unstable intervals  is at least $\delta N$. We call $f$ not exceptional otherwise.  

Let $\CE_e = \CE_e(R,\al,\beta; \delta)$ denote the set of vectors $\Bv_f$ associated to exceptional function $f$. Our goal in this section is the following.
\begin{theorem}\label{thm:exceptional:W} Assume that $\al,\beta,\delta$ satisfy 
\begin{equation}\label{eqn:parametersthm}
\al \asymp \delta^{2}, \quad \quad \beta \asymp \delta^{3/4},\quad \quad   N^{-1/6} < \delta < (\log n)^{-16}.
\end{equation}
Assume that $G$ has form \eqref{eqn:G:W}. Then 
$$\P\Big(\Bv_G \in \CE_e\Big) \le 2 e^{-\delta^5 N/ 4C_0\log^{8} n}+ e^{-\Theta(N/\log^2 n)}.$$
\end{theorem}

Our proof is somewhat similar the the orthogonal and Elliptic models, except that here we have to perturb functions $f$ from $\CE_{W,d_0}$ by functions $g$ from $\CE_{W,3}$ (also defined in Lemma \ref{lemma:W:prob:d}). More details follow below. 

First assume that $f$ (playing the role of $G$) is exceptional and $\Ba=\Bv_f \in \CE_{W,d_0}$. Then there are $K=\lfloor \delta N/3 \rfloor$ unstable intervals that are $4/N$-separated.

 Now for each unstable interval in this separated family we choose $x_j \in 3 I_j$ where  $|f(x_j)|\le \al $ and $|f'(x_j)|\le \beta N$ and consider the interval $B(\theta_j, \gamma/N)$ for some $\gamma <1$ chosen sufficiently small (given $\delta$, for instance $\gamma \asymp \delta^{5/4}$ would suffice). Let 
$$M_j:= \max_{\theta\in B(\theta_j,\gamma/N)} |f''(\theta)|.$$  
By Lemma \ref{lemma:largesieve} (with $S= \log^4 n$), as $\Ba\in \CE_{W,d_0}$ (note that here we just need the bounds from \eqref{MB:W} for $d=0,1,2,3$, that is $\Ba\in \CE_{W,3}$) we have 
$$\sum_{j=1}^K M_j^2 \le  (\log^4 n) (C^\ast N) (C^\ast N)^4.$$

We thus infer from the above that the number of $j$ for which $M_j \ge C_1\delta^{-1/2} (C^\ast)^{5/2} (\log^2 n) N^2$ is at most $C_1^{-2}\delta N$. Hence for at least $(1/3 -C_1^{-1})\delta N$ indices $j$ we must have $M_j <C_1\delta^{-1/2} (C^\ast)^{5/2} (\log^2 n) N^2$.

Consider our function over $B(\theta_j, \gamma/N)$, then by Taylor expansion of order two around $\theta_j$, we obtain for any $\theta$ in this interval
$$ |G(\theta)| \le \al + (\beta N) \gamma/N + [C_1\delta^{-1/2} (C^\ast)^{5/2} (\log^2 n) N^2] \gamma^2 /N^2 =  \al + \beta  \gamma + C_1\delta^{-1/2} (C^\ast)^{5/2}(\log^3 n) \gamma^2 $$
and
$$  |G'(\theta)| \le \beta N+  [C_1\delta^{-1/2} (C^\ast)^{5/2} (\log^3 n)  (\log^2 n) N^2]  \gamma N /N = (\beta + C_1\delta^{-1/2} (C^\ast)^{5/2} (\log^2 n) \gamma) N.$$
Now given $\tau>0$ and consider a function $g = G_\Bb=\sum_{i=0}^n b_i  e^{-(x/N)^2/2} \frac{(x/N)^i}{\sqrt{i!}}$ such that 
\begin{equation}\label{MB:W:g}
\int_T (g^{(d)}(x))^2 dx \le  \tau^2   (C^\ast N)^{2d},\quad \quad 0\le d\le 3.
\end{equation}
In other words, in the unscaled version $F=F_\Bb$ we have  $\int_I F_{}^{(d)}(x)^2 dx \le  \tau^2   (C^\ast)^{2d} N 0\le d\le 3$.
By Lemma \ref{lemma:W:d}, this can be attained (here $C^\ast = \log^3 n$) if (with room to spare) 
\begin{equation}\label{eqn:gb}
\|g\|_2^2 :=\sum_i b_i^2 \le \tau^2 N/\log^{8}n.
\end{equation}
Then as the intervals $B(\theta_j, \gamma/N)$ are $4/N$-separated, by Lemma \ref{lemma:largesieve} implies that
$$\sum_j \max_{x \in B(\theta_j, \gamma/N)} g(\theta)^2   \le \tau^2 C^\ast N$$
and 
$$\sum_j \max_{\theta \in B(\theta_j, \gamma/N)} g'(\theta)^2   \le  \tau^2  (C^\ast N) (C^\ast N)^2.$$
Hence, again by an averaging argument, the number of intervals where either $\max_{\theta \in B(\theta_j, \gamma/N)} |g(\theta)| \ge C_2  \delta^{-1/2} (C^\ast)^{1/2} \tau^{}$ or   $\max_{\theta \in B(\theta_j, \gamma/N)} |g'(\theta)| \ge C_2 \delta^{-1/2} (C^\ast)^{3/2}\tau^{} N$ is bounded from above by $2 C_2^{-2}\delta N$. 

On the remaining at least $(1/3 - C_1^{-2} - 2 C_2^{-2})\delta N$ intervals, with $h=f+g$, we have simultaneously that
$$|h(\theta)| \le  \al + \beta  \gamma + C_1\delta^{-1/2} (C^\ast)^{5/2}  (\log^2 n) \gamma^2 + C_2  \delta^{-1/2} (C^\ast)^{1/2} \tau^{} $$
and
$$|h'(\theta)| \le  (\beta + C_1\delta^{-1/2} (C^\ast)^{5/2}  (\log^2 n) \gamma+  C_2 \delta^{-1/2} (C^\ast)^{3/2}\tau^{}) N.$$
For short, let 
$$\al':= \al + \beta  \gamma + C_1\delta^{-1/2} (C^\ast)^{5/2}  (\log^2 n) \gamma^2 + C_2  \delta^{-1/2} (C^\ast)^{1/2} \tau^{1/2} $$
 and  \
 $$\beta':=\beta + C_1\delta^{-1/2} (C^\ast)^{5/2}  (\log^2 n) \gamma+  C_2 \delta^{-1/2} (C^\ast)^{3/2}\tau^{1/2}.$$
 Note that the parameters here are almost the same as in the Elliptic model case, except the powers of $\log n$ and the powers of $C^\ast$.
  
It follows that $\Bv_h$ belongs to the set $\CU=\CU(\al, \beta,\gamma,\delta, \tau)$ in $\Omega^{n+1}$ of the vectors  corresponding to $h$, for which the measure of $x$ with $|h(x)| \le  \al'$ and $|h'(x)| \le  \beta' N $ is at least $2(1/3 - C_1^{-2} - 2 C_2^{-2})\delta \gamma$ (because this set of $x$ contains $(1/3 - C_1^{-2} - 2 C_2^{-2})\delta N$ intervals of length $2\gamma/N$). 
Putting together, together with \eqref{eqn:gb} we have obtained the following claim.

\begin{claim}
\label{claim2} Assume that $\Bv_{G} \in \CE_e$. Then for any $g$ with $\|g\|_2^2 \le \tau^2 N/\log^8 n$ we have $\Bv_{G+g} \in \CU$. 
\end{claim}

We next argue as in the proof of \eqref{eqn:U:O} and \eqref{eqn:U:E}, by relying on Theorem \ref{thm:repulsion:W} and by choosing the parameters so that $\al' \beta'$ is much smaller than $\delta \gamma$ (for instance $\gamma \asymp \delta^{5/4}, \tau \asymp \delta^{5/2}$ and $N^{-1/6} < \delta < (\log n)^{-16}$) we have 
\begin{equation}\label{eqn:U:W}
\P(\Bv_{G} \in \CU) \le 1/2.
\end{equation}

\begin{proof}(of Theorem \ref{thm:exceptional:W}) By Theorem \ref{thm:sobolev} and \eqref{eqn:U:E}, 
$$\P(\Bv\in \CE_e) \le 2 e^{-(\tau \sqrt{N}/ \log^4 n)^2/4C_0}.$$
\end{proof}

\subsection{Roots over unstable intervals and proof completion}\label{subsection:lowertail:W} We next proceed as in Subsection \ref{subsection:lowertail:E}. Let $0<\eps<1$ be given.  For sufficiently large $C_1,C_2$, let $\delta, \la$ be chosen so that 
\begin{equation}\label{delta:eps:W}
\delta =O(\min\{\eps/\log(1/\eps),\eps/(\log n)^{C_1}\}); \quad \quad  \la=\delta^{C_2}.
\end{equation}
Then with $\al , \beta$ be chosen as in Theorem \ref{thm:exceptional:W} (and with $C^\ast = O(\log^3 n)$), we clearly have 
$$\la (C^\ast)^{1/2} \le \al,  \quad \la (C^\ast)^{3/2} \le \beta.$$
With these parameters, with $f\in \CE_{W,d_0}$ (defined in Lemma \ref{lemma:W:prob:d}), the conditions in Proposition \ref{prop:manyroots} are satisfied (because of \eqref{eqn:W_0}). 
Hence we obtain the following consequence (similarly to Corollary \ref{cor:manyroots:O} and  Corollary \ref{cor:manyroots:E}).

\begin{cor}\label{cor:manyroots:W} With the parameters as above, a non-exceptional $f \in \CE_{W,d_0}$  cannot have more than $\eps N/2$ roots over any $\delta N$ intervals $I_i$. In particularly, $f$ cannot have more than $\eps N/2$ roots over the unstable intervals. 
\end{cor}
As a consequence, we also have the following analog of Corollary \ref{cor:control:E}.

\begin{cor}\label{cor:control:W} With the same parameters as in  Proposition \ref{prop:manyroots}, with $C^\ast = \log^2 n$ and $N=\sqrt{n}$ we have

$$\P\Big(N_s(G) \1_{G \in \CE_e^c \cap  \CE_{W,d_0}} \le \E N_T(G)- \eps N \Big)=o(1)$$
and
$$\E \Big(N_s(G) \1_{F \in \CE_e^c  \cap  \CE_{W,d_0} }\Big) \ge \E N_T(G)- 2 \eps N/3.$$
\end{cor}

From this point on, the proof of Theorem \ref{thm:Weyl} is almost identical to that of Theorem \ref{thm:Elliptic}, where the only difference in the treatment for both the lower tail and upper tail is that we need to use Theorem \ref{thm:W:gen} (where $\R$ is replaced by $I=[\frac{1}{C_W} \sqrt{n}, {C_W} \sqrt{n}]$) instead of  Theorem \ref{thm:E:gen}, and Theorem \ref{thm:sobolev} instead of  Theorem \ref{thm:E:concentration}.

\section{Proof of Theorem \ref{thm:dev}}
We first provide an analog of Proposition \ref{prop:manyroots} where there are too many roots.

\begin{proposition}[overcrowding estimate]\label{prop:overcrowding} Let $A_0>0$ be a given constant. Let $C^\ast, N$ be given parameters, where $C^\ast$ might depend on $N$. Assume that $G$ is a smooth real valued function over $T=[0,1]$ where $G$ has at most $N^{A_0}$ roots, and for any $1\le d  \le A_0 \log N$ we have
\begin{equation}\label{eqn:MB:final}
\int_T (G^{(d)}(x))^2 dx   \le (C^\ast N)^{2d}  \mbox{ and }  \int_{T} G(x)^2 dx \le (C^\ast)^2.
\end{equation}
Assume that $T$ has $AN$ roots where $A/C^\ast$ is sufficiently large. Then
$$\max_{x\in T}|G(x)| \le 2^{-A/8+2}  (C^\ast)^{1/2}\mbox{ and } \max_{x\in T} |G'(x)| \le 2^{-A/8+3} (C^\ast)^{3/2}  N.$$ 
\end{proposition}
In other words, in the above overcrowding asumption we can obtain useful upper bounds for $|G(\cdot)|$ and $|G'(\cdot)|$ over the whole interval $T$. 

\begin{proof}(of Proposition \ref{prop:overcrowding}) Our proof is almost identical to that of Proposition \ref{prop:manyroots}, so we'll be brief. Among the $N/4$ intervals of length $4/N$ we first throw away those of less than $A/8$ roots, hence there are at least $A N/2$ roots left. For convenience we denote the remaining intervals by $J_1,\dots, J_M$ and let $m_1,\dots, m_M$ denote the number of roots over each of these intervals respectively.

Argue as in the proof of  Proposition \ref{prop:manyroots}, in the next step we again expand the intervals $J_j$ to larger intervals $\bar J_j$ of length $\lceil c m_j/4 \rceil \times (4/N)$ for some small $c$, such as $c=1/(16C^\ast e)$. Furthermore, if the expanded intervals $\bar J_{i_1}',\dots, \bar J_{i_k}'$ of $\bar J_{i_1},\dots, \bar J_{i_k}$ form an intersecting chain, then we create a longer interval
 $\bar J'$ of length $\lceil c(m_{i_1}+\dots+m_{i_k})/R \rceil \times (R/N)$, which contains them and therefore contains at least $m_{i_1}+\dots+m_{i_k}$ roots. After the merging process,
 we obtain a collection 
 $\bar J_1',\dots, \bar J_{M'}'$ with the number of roots $m_1',\dots, m_{M'}'$ respectively, so that $\sum m_i'\geq A N/4$.
  Note that now $\bar J_i'$ has length $\lceil cm_i'/4 \rceil \times (4/N) \approx cm_i'/N$. 

Next, consider the sequence $d_l:=2^l A/16, 0\le l\le \log (n^{A_0}/A)$. We classify the sequence $\{m_i'\}$ into groups $G_l$ where 
$$d_l \le m_i' < d_{l+1}.$$ 
Assume that each group $G_l$ has $k_l=|G_l|$ distinct extended intervals. As each of these intervals has between $d_l$ and $d_{l+1}$ roots, we have 
$$\sum_l k_l d_l \ge \sum_i m_i'/2 \ge A N/8.$$
For $\la = 2^{-A/8}$, we call an index $l$ {\it bad} if 
$$(1/2)^{d_{l}}  (N/2k_{l})^{1/2}  \ge  \la.$$ 
(In other words, that is when $k_l \le \frac{N}{2\la^2 4^{d_{l}}}$.) The total number of roots over the intervals corresponding to bad indices can be bounded by 
$$\sum_i m_i' \le \sum_l k_l d_{l+1}   \le \frac{N}{2\la^2} \sum_{l=0}^\infty\frac{2 d_{l}}{ 4^{d_{l}}} \le \frac{N}{\la^2 2^{A/4}} \le N.$$

Now consider the collection $G_{l}$ of each good index $l$. Notice that these intervals have length approximately between $cd_l/N$ and $2cd_l/N$. Let $I$ be an interval among the $k_l$ intervals in $G_l$. By the definition of $I$ and by Lemma \ref{lemma:expansion} we have 
\begin{equation}\label{eqn:PP_d}
\max_{x \in I} |G(x)| \le (\frac{1}{2C^\ast})^{d_l} (\frac{1}{N})^{d_l} \max_{x\in I} |G^{(d_l)}(x)| \le \frac{\la}{ (N/2k_{l})^{1/2}  } (\frac{1}{C^\ast N})^{d_l} \max_{x\in I} |f^{(d_l)}(x)|
\end{equation}
as well as 
\begin{equation}\label{eqn:PP_d'}
\max_{x \in I} |G'(x)| \le N \times (\frac{1}{2C^\ast })^{d_l-1} (\frac{1}{N})^{d_l}  \max_{x\in I} |G^{(d_l)}(x)|  \le N \times \frac{2\la C^\ast}{ (N/2k_{l})^{1/2}  } (\frac{1}{C^\ast N})^{d_l} \max_{x\in I} |G^{(d_l)}(x)|.
\end{equation}
On the other hand, as these $k_l$ intervals are $4/N$-separated, by the assumption of our proposition and by Lemma \ref{lemma:largesieve}  we have 
$$\sum_{\bar J_i'\in G_l}\max_{x\in \bar J_{i}'}(G^{(d_{l})}(x))^2 \le 4 (C^\ast N)^{2d_l+1}.$$
Hence we see that for at least half of the intervals $J_i'$ in $G_l$ 
$$\max_{x\in J_{i}'}|G^{(d_{l})}(x)| \le 4 (C^\ast N/k_{l})^{1/2} (C^\ast N)^{d_l} .$$ 
It follows from \eqref{eqn:PP_d} and \eqref{eqn:PP_d'} that over these intervals
$$\max_{x \in J_{i}'} |G(x)| \le  \frac{\la}{ (N/2k_{l})^{1/2}  } (\frac{1}{C^\ast N})^{d_l}  4 (C^\ast N/k_{l})^{1/2} (C^\ast N)^{d_{l}}  \le  4\la (C^\ast)^{1/2}$$
and similarly,
$$\max_{x \in J_i'} |G'(x)| \le N \times \frac{2 \la C^\ast}{ (N/2k_{l})^{1/2}  } (\frac{1}{C^\ast N})^{d_l} 4 (C^\ast N/k_{l})^{1/2} (C^\ast N)^{d_{l}}   \le 8 \la (C^\ast)^{3/2} N.$$

Letting $A_l$ denote the union of all such intervals $J_{i}'$ of a given good index $l$, and letting $\CA$ denote the union of the $A_l$'s over all  good indices $l$, we obtain that if $\CA$ is not the whole $T$ then (with $\mu$ denoting Lebesgue measure)
\begin{align*}
\mu(\CA) \ge \sum_{l, \textsf{good}} (cd_l/N) k_l/2& \ge \sum_{l, \textsf{good}} (c/4) d_{l+1}k_l/N \ge \sum_{l, \textsf{good}} (c/4) m_l'/N \\
& \ge (c/4)(A N/8 -N)/N \ge \frac{A}{1024 C^\ast e} > 1, 
\end{align*}
a contradiction if $A$ is sufficiently large. Hence the set $\CA$ must be the whole $T$. Finally, notice that over $\CA$ we have $\max_{x \in T} |G(x)| \le 4\la (C^\ast)^{1/2}$ and $\max_{x \in T} |G'(x)|\le 8 \la (C^\ast)^{3/2}  N$. 
\end{proof}

\begin{proof}(of Theorem \ref{thm:dev}) Note that under $\CE_{b,E}$ (from Proposition \ref{prop:manyroots:E}) in the Elliptic case and $\CE_{W,d_0}$ (from Lemma \ref{lemma:W:prob:d}) in the Weyl case that the Condition \eqref{eqn:MB:final} is satisfied. Hence if $N_I \ge A \sqrt{n}$ for some $A \gg \log^4 n$ (we can choose $C^\ast = \log^3 n$ in both cases) then we can apply Proposition \ref{prop:overcrowding} (to $G$ of the form \eqref{eqn:G:E'} or \eqref{eqn:G:W}, and where $T=[1/C_E, \pi - 1/C_E]$ or $T= \pm [1/C_W^2,1]$  
depending on the model) to conclude that $|F(x)| \le  2^{-A/8+2} \log^{3/2} n \le 2^{-A/16}$ for all $x \in I$, where $F$ takes the form \eqref{eqn:F:E} or \eqref{eqn:F:W} in the Elliptic or Weyl case respectively.

Hence it suffices to show that for some $x_0 \in I_E$ or $I_W$ (depending on the model) we have 
$$\P(|F(x_0)| \le 2^{-A/16}) \le \exp(-c A).$$
We first use the following small ball estimate for lacunary sequence from \cite[Lemma 9.2]{TV} (see also \cite[Lemma 7]{NgNgV} for related variants).
\begin{lemma} Assume that $v_1,\dots, v_n$ be real numbers and such that there is a subsequence $v_{i_1},\dots, v_{i_m}$ with the property that 
$$|v_{i_j}| \ge 2 |v_{i_{j+1}}|.$$
Then with $\xi_i$ as in Theorem \ref{thm:dev}, there exists a constant $c$ such that
$$\sup_x \P(|\xi_1v_1+\dots+\dots+\xi_n v_n - x| \le |v_{i_m}|) =O(\exp(-c m)).$$
\end{lemma}
Next, in the Weyl case we can choose $|x_0| = \sqrt{n}/2$. Note that 
$$|\frac{x_0^{i+1}}{\sqrt{(i+1)!}}/ \frac{x_0^{i}}{\sqrt{i!}}| = |\frac{x_0}{\sqrt{i+1}}|.$$ 
By using Claim \ref{claim:Weyl:ix} one can easily extract a subsequence $v_{i_j}$ from $\{e^{-x_0^2/2}\frac{|x_0|^{i+1}}{\sqrt{(i+1)!}}\}_{i=0}^n$ so that $v_{i_m} \approx 2^{-A/16}$ and 
$$2 |v_{i_{j+1}}| \le  |v_{i_j}| \le 4  |v_{i_{j+1}}| .$$
Similarly, in the Elliptic case we can choose $x_0$ to have order $\sqrt{n}$, noting that
$$ \Big |\sqrt{\binom{n}{i+1}} (x_0/\sqrt{n})^{i+1}/  \sqrt{\binom{n}{i}} (x_0/\sqrt{n})^i \Big| =|x_0/\sqrt{n}| \sqrt{(n-i-1)/(i+1)}.$$
From here one can easily extract a lacunary subsequence from $\{ \frac{1}{(1+x_0^2/n)^{n/2}} \sqrt{\binom{n}{i}} (x_0/\sqrt{n})^i  \}_{i=0}^n$ as in the Weyl case, we omit the details.
\end{proof}


\appendix


\begin{thebibliography}{99}


\bibitem{AL} M. Ancona and T. Letendre, Roots of Kostlan polynomials: moments, strong Law of Large Numbers and Central Limit Theorem, Annales Henri Lebesgue 4 (2021), 1659--1703. 
\vskip .05in

\bibitem{AA}  A. I. Aptekarev, A. Draux, V. A. Kalyagin and D. N. Tulyakov, Asymptotics  of  sharp  constants  of  Markov-Bernstein  inequalities  in  integralnorm  with  Jacobi  weight.Proc.  Amer.  Math.  Soc.,143(2015),  3847-3862.
\vskip .05in

 \bibitem{AS} F. Aurzada, T. Simon. Persistence probabilities and exponents. Lévy matters. V, 183–224, Lecture
Notes in Math., 2149, Lévy Matters, Springer, Cham, 2015.
\vskip .05in

 \bibitem{ADL} J.M. Aza\"is, F. Dalmao, J. Le\'on, CLT for the zeros of classical random trigonometric polynomials. Ann. Inst. Henri-Poincare. 52(2) (2016), 804-820.
  \vskip .05in
\bibitem{AL}  J. M. Aza\"is, J. Le\'on, CLT for crossings of random trigonometric polynomials. Electron. J. Probab. 18 (68) (2013), 1-17.
  \vskip .05in
  
  \bibitem{AP'} J. Angst, G. Poly, Variations on Salem–Zygmund results for random trigonometric polynomials: application to almost sure nodal asymptotics, Electron. J. Probab. 26: 1-36 (2021). 
  \vskip .05in

   \bibitem{BCP} V. Bally, L. Caramellino, and G. Poly, Non universality for the variance of the number of real roots of random trigonometric polynomials, Probab. Theory Relat. Fields (2018). https://doi.org/10.1007/s00440-018-0869-2. 
  \vskip .05in

\bibitem{BL} S. Bobkov and M. Ledoux, On modified logarithmic Sobolev inequalities for Bernoulli and Poisson measures, J. Funct. Anal. 156 (1998), no. 2, 347-365. 
  \vskip .05in


  \bibitem{BDFZ} R. Basu, A. Dembo, N. Feldheim and Ofer Zeitouni, Exponential concentration for zeroes of stationary Gaussian processes, \url{arxiv.org/abs/1709.06760}, to appear in International Mathematics Research Notices.
    \vskip .05in


\bibitem{BEbook} P. Borwein and T. Erd\'elyi, Polynomials and Polynomial Inequalities, Graduate texts in Mathematics, Springer (1995), Berlin--New York.
 \vskip .05in

\bibitem{CNg} H. Can and O. Nguyen, Concentration inequalities for the number of real zeros of Kac polynomials, \url{https://arxiv.org/abs/2311.15446}.

    \vskip .05in


\bibitem{CP} H. Can and V.-H. Pham, Persistence probability of random Weyl polynomial,  J Stat Phys 176, 262--277 (2019) .

    \vskip .05in

 \bibitem{Cuz} J. Cuzick, A central limit theorem for the number of zeros of a stationary Gaussian process, Ann. Probab. 4 (1976): 547--56.    
      \vskip .05in  

    
\bibitem{Dal}    F. Dalmao, Asymptotic variance and CLT for the number of zeros of Kostlan–Shub–Smale random polynomials, C. R. Math. Acad. Sci. Paris 353 (2015), no. 12, 1141-1145.
  \vskip .05in  
  
  
\bibitem{GaWe} Damien Gayet and Jean-Yves Welschinger. Exponential rarefaction of real curves with many components. Publ. Math. Inst. Hautes Etudes Sci., (113):69-96, 2011.

\vskip .05in

\bibitem{DM} A. Dembo and S. Mukherjee, Nozero-crossings for random polynomials and the heat equation, Ann.Probab.
43(1), 85--118 (2015).
\vskip .05in



\bibitem{DM'}  A. Dembo and S. Mukherjee, Persistence of Gaussian processes:non-summable correlations. Probab.The-
ory Relat. Fields 169(3–4), 1007--1039 (2017).

\vskip .05in



\bibitem{DPSZ} A. Dembo, B. Poonen, Q-M. Shao, O. Zeitouni. Random polynomials having few or no real zeros. J.
Amer. Math. Soc. 15, no. 4 (2002): 857--892.
\vskip .05in

  \bibitem{DL} D. Diatta and A. Lerario, Low degree approximation of random polynomials, Foundations of Computational Mathematics (2022) 22:77--97.

\vskip .05in

 \bibitem{DLNgNgP} Y. Do, D. Lubinski, H. Nguyen, O. Nguyen, and I. Pritsker. Central Limit theorem for real roots of gaussian orthogonal polynomials in the bulk, to appear in , Annales de l’Institut Henri Poincare.
  \vskip .05in
  
\bibitem{DNh}  Yen Do and N. Nguyen, Real roots of random polynomials with coefficients of polynomial growth: asymptotics of the variance, \url{https://arxiv.org/abs/2303.05478}.
  \vskip .05in
  
\bibitem{DV}  Do, Yen; Vu, Van Central limit theorems for the real zeros of Weyl polynomials. Amer. J. Math. 142 (2020), no. 5, 1327--1369. 

\vskip .05in



\bibitem{DNgNg} Y. Do, H. Nguyen and O. Nguyen, Random trigonometric polynomials: universality and non-universality of the variance for the number of real roots (with Yen Do and Oanh Nguyen, Annales de l’Institut Henri Poincare, Vol. 58, No. 3 (2022), 1460--1504).
\vskip .05in

  
\bibitem{DONgV1} Y. Do, O. Nguyen, and V. Vu, Roots of random polynomials with coefficients of polynomial growth, Ann. Probab., 46, no. 05 (2018), 2407--2494.

\vskip .05in
\bibitem{DONgV2} Y. Do, O. Nguyen, V, Vu, Real zeros of random orthogonal polynomials, Transactions of the American Mathematical Society 376(09), 2023, 6215-6243.
\vskip .05in


\bibitem{Dun} J. Dunnage, The number of real zeros of a class of random algebraic polynomials.
Proc. London Math. Soc. (3)18(1968), 439--460.
\vskip .05in


\bibitem{EK} A. Edelman and E. Kostlan, How many zeros of a random polynomial are real?, Bull. Amer. Math. Soc., 32, no. 1 (1995), 1--37.
\vskip .05in

\bibitem{EO} P. Erd\H{o}s, A. C. Offord, On the number of real roots of a random algebraic equation, Proc. London Math. Soc.
6 (1956), 139--160.
\vskip .05in

\bibitem{FFN} N. Feldheim, O. Feldheim, S. Nitzan, Persistence of Gaussian stationary processes: A spectral perspective, Ann. Probab. 49(3): 1067--1096 (May 2021).
\vskip .05in

\bibitem{FFN'} N. Feldheim, O. Feldheim, B. Jaye, F. Nazarov, and S. Nitzan, On the probability that a stationary Gaussian process with spectral gap remains non-negative on a long interval. Int. Math. Res. Not., 2018, rny248. 
\vskip .05in


 \bibitem{Ga} P.X. Gallagher, The large sieve, Mathematika 14 (1967), 14-20.
  \vskip .05in

\bibitem{LGass} L. Gass, Cumulants asymptotics for the zeros counting measure of real Gaussian processes, Electron. J. Probab. 28: 1--45 (2023). 
  \vskip .05in



\bibitem{GW} A. Granville and I. Wigman. The distribution of the zeros of random trigonometric polynomials. Amer. J. Math. 133 (2) (2011) 295-357.
 \vskip .05in

\bibitem{Halasz} G. Hal\'asz, Estimates for the concentration function of combinatorial number theory and probability, Period. Math. Hungar. 8 (1977), no. 3-4, 197-211.
 \vskip .05in
 
 \bibitem{book} J. Hough, M. Krishnapur, Y. Peres, and B. Vir\'ag, Zeros of Gaussian analytic functions and determinantal point processes, volume 51. American Mathematical Society Providence, RI, 2009.
 \vskip .05in

 
 
\bibitem{KZ} Z. Kabluchko and D. Zaporozhets,
Asymptotic distribution of complex zeros of random analytic functions, Ann. Probab. 42 (2014), no. 4, 1374--1395. 
\vskip .05in

\bibitem{KZ1} Z. Kabluchko, D. Zaporozhets, Universality for zeros of random analytic functions, \url{https://arxiv.org/abs/1205.5355}.
 \vskip .05in

\bibitem{Kac} M. Kac, On the average number of real roots of a random algebraic equation, Bull. Amer. Math. Soc, 49, no. 1 (1943), 314--320.
\vskip .05in

\bibitem{K} E. Kostlan, On the distribution of roots of random polynomials, From topology to computation: proceedings of the Smalefest (Berkeley, CA, 1990), Springer, New York, 1993, pp. 419431.

\vskip .05in

\bibitem{KK} M. Krishna and M. Krishnapur. Persistence probabilities in centered, stationary, Gaussian processes in discrete time. Indian J. Pure Appl. Math. 2016, 47 183--194
\vskip .05in


\bibitem{L} M. Ledoux, The concentration of measure phenomenon, Mathematical Surveys and Monographs 89, AMS (2001), Providence, RI.
\vskip .05in
\bibitem{LPX} D Lubinsky, I Pritsker, and X Xie. Expected number of real zeros for random linear combinations of orthogonal
polynomials. Proceedings of the American Mathematical Society, 144(4):1631–1642, 2016.

\vskip .05in
\bibitem{LP} D. S. Lubinsky and I. E. Pritsker, Variance of real zeros of random orthogonal polynomials, \url{https://arxiv.org/abs/2101.06727}.
\vskip .05in

\bibitem{Mas1} N. B. Maslova. On the variance of the number of real roots of random polynomials. Theory of Probability \& Its Applications, 19(1):35–52, 1974.
\vskip .05in

\bibitem{Mas2} N. B. Maslova. On the distribution of the number of real roots of random polynomials. Theory of Probability \& Its Applications, 19(3):461–473, 1975.
\vskip .05in



\bibitem{NS} F. Nazarov, M. Sodin, On the number of nodal domains of random spherical harmonics. Amer. J. Math. 131 (2009), 1337-1357.
\vskip .05in
\bibitem{NS-general} Nazarov, F. and M. Sodin. Asymptotic laws for the spatial distribution and the number of connected components of zero sets of Gaussian random functions.? Journal of Mathematical Physics, Analysis, Geometry 12, no. 3 (2016): 205-78.
 \vskip .05in
\bibitem{NgNgV} H. Nguyen, O. Nguyen and V. Vu, On the number of real roots of random polynomials, Communications in Contemporary Mathematics (2016) Vol. 18, 4, 1550052.
 \vskip .05in

\bibitem{NgV-survey} H. Nguyen and V. Vu, Small probability, inverse theorems, and applications, Paul Erd\H{o}s' 100th anniversary, Bolyai Society Mathematical Studies, Vol. 25 (2013).
\vskip .05in
\bibitem{NgZ} H. Nguyen and O. Zeitouni, Exponential concentration for the number of roots of random trigonometric polynomials, Annales de l’Institut Henri Poincar\'e, to appear.
. 
\vskip .05in

 \bibitem{Ngcurve} H. Nguyen, Concentration of the number of intersections of random eigenfunctions on flat tori, Proceedings of the American Mathematical Society 151 (2023), no. 7, 3127-3143. 
 \vskip .05in
 
 
\bibitem{NgV}  H. Nguyen and V. Vu, Optimal inverse Littlewood-Offord theorems, Adv. Math., Vol. 226 6 (2011), 5298--5319.
\vskip .05in

\bibitem{OV} O. Nguyen and V. Vu, Roots of random functions: A general condition for local universality, American Journal of Mathematics 144(01), 2022, 1-74.
\vskip .05in
\bibitem{OV-CLT} O. Nguyen and V. Vu, Random polynomials: central limit theorems for the real roots, Duke Math. J. 170 (2021), no. 17, 3745–3813. 

\vskip .05in


\bibitem{Qualls} C. Qualls, On the number of zeros of a stationary Gaussian random trigonometric polynomial, J. London Math. Soc. (2) 2 (1970), 216-220.
\vskip .05in

\bibitem{Rozen} Y. Rozenshein, The Number of Nodal Components of Arithmetic Random Waves. Int. Math. Res. Not. IMRN 2017, no. 22, 6990-7027.
\vskip .05in


\bibitem{RV-rec} M.~Rudelson and R.~Vershynin, Smallest singular value of a random rectangular matrix, {\it Communications on Pure and Applied Mathematics},  \textbf{62} (2009), 1707-1739.
\vskip .05in


 \bibitem{SM} G. Schehr and S. Majumdar, Real Roots of Random Polynomials and Zero Crossing Properties of Diffusion Equation,  J. Stat. Phys., 132 (2008), 235-273.
\vskip .05in

\bibitem{TV} T. Tao and V. Vu, Local universality of zeroes of random polynomials, Int. Math. Res. Not. IMRN, (2014), 0-84.




\end{thebibliography}
\end{document}